\documentclass[a4paper,11pt,twoside]{amsart}
\usepackage[latin1]{inputenc}
\usepackage[T1]{fontenc}

\usepackage{a4wide}

\usepackage{ifpdf}
\ifpdf\usepackage[pdftex]{hyperref}
\else\usepackage[hypertex]{hyperref}\fi

\usepackage{slashed}

\theoremstyle{plain}
\newtheorem{thm}{Theorem}[section]
\newtheorem{prop}[thm]{Proposition}
\newtheorem{lemma}[thm]{Lemma}
\newtheorem{cor}[thm]{Corollary}

\theoremstyle{definition}
\newtheorem{defn}[thm]{Definition}

\theoremstyle{remark}
\newtheorem{rem}[thm]{Remark}
\newtheorem{rems}[thm]{Remarks}

\theoremstyle{remark}

\DeclareMathOperator{\tr}{Tr}

\DeclareMathOperator{\id}{id}
\DeclareMathOperator{\spec}{Spec}

\DeclareMathOperator{\ind}{ind}
\DeclareMathOperator{\ch}{ch}

\let\dsp=\displaystyle

\def\R{\mathbb R}
\def\C{\mathbb C}
\def\N{\mathbb N}
\def\Z{\mathbb Z}
\def\c{\mathbf c}

\begin{document}

\title[Spectral invariants on contact manifolds]{Topological and
  dynamical aspects \\of some spectral invariants of contact manifolds
  \\with circle action}

\author{Michel Rumin}
\address{Laboratoire de Mathématiques d'Orsay\\
  Université Paris-Saclay\\
  91405 Orsay Cedex\\ France}

\email{michel.rumin@universite-paris-saclay.fr}

\date{\today}

\begin{abstract}
  We study analytic torsion and eta like invariants on CR contact
  manifolds of any dimension admitting a transverse circle action, and
  equipped with a unitary representation. We show that, when defined
  using the spectrum of relevant operators arising in this geometry,
  the spectral series involved can been interpreted in their whole,
  both from a topological viewpoint, and as purely dynamical functions
  of the Reeb flow.
\end{abstract}




\maketitle



\emph{\small Tous les chemins m\`enent \`a Rome.}

\section{Introduction}
\label{sec:introduction}

This paper deals with the study of some spectral series associated to
geometric invariants on particular compact contact CR manifolds
$M$. Those who admit a transverse locally free circle action.

That means the generator $T$ of this action is the Reeb field of an
invariant contact form $\theta$ and preserves an integrable complex
structure $J$ on $H = \ker \theta$. In that case, the orbifold
$N= M/ \mathbb{S}^1$ appears to be a K\"ahler $V$-manifold in the
sense of Satake \cite{Kawasaki}. It is a topological space endowed
with a smooth open dense K\"ahler structure, corresponding to the
generically free orbit, and a finite number of singular conical
points, corresponding to the exceptional fibers.

Independently of this circle action, we also equip $M$ with a unitary
representation $\rho : \pi_1(M) \rightarrow U(d)$. This broadens the
framework to twisted spectral invariants associated to the flat
bundles of these representations, and provides us with dynamical data
using the holonomies induced by the closed orbits of the Reeb flow.

\subsection{Around the contact analytic torsion}
\label{sec:around-cont-analyt}

\ \smallskip

We will be concerned with spectral series associated to two typical
spectral invariants. The first one is the `contact' analytic torsion,
as defined in \cite{Rumin-Seshadri}. This analytic determinant is
associated to the contact de Rham complex $(\mathcal{E}^*, d_Q)$, a
hypoelliptic complex, homotopic to the usual Hodge-de Rham one, but
benefiting from better contact homogeneity when rescaling $\theta$ in
$k \theta$. See Section \ref{sec:contact-complex} for a presentation
of this construction. This resolution of constants starts on functions
with $d_Q = d_H$, the usual differential, but restricted to the
horizontal vectors $H$ in the contact distribution. The price for this
however, is the appearance of a second order differential
$D= d_Q : \mathcal{E}^n \rightarrow \mathcal{E}^{n+1}$ in 'middle
degree', with $\dim N = 2n+1$. In order to preserve homogeneity, this
in turns leads to using \emph{fourth-order} Laplacians
$\mathbf{\Delta}_Q$ in all degrees; see
Section~\ref{sec:middle-degree-case}. These self-adjoint operators are
hypoelliptic. They possess discrete spectrum and smooth heat kernels
on compact contact manifolds.

\smallskip

This allows to consider our first spectral series which is related to
the analytic torsion of the contact complex. In the Riemannian
setting, the analytic torsion was introduced by Ray and Singer in
\cite{RS} as an infinite dimensional analogue of the
Reidemeister-Franz torsion of a finite dimensional complexes. It is
defined by an appropriate combination of analytic determinants of the
Hodge-de Rham Laplacians using their zeta functions.

In \cite{Rumin-Seshadri}, the authors proposed to adapt the
construction on contact manifolds. Starting from heat kernels, one
considers for $t > 0$
\begin{equation}
  \label{eq:1}
  \vartheta(t) = \sum_{k=0}^n (-1)^k (n+1-k) \tr (e^{-t\mathbf{\Delta}_Q}\mid
  \mathcal{E}^k). 
\end{equation}

This particular combination leads to the definition of the analytic
torsion of the contact complex. Briefly, taking Mellin transform leads
to zeta functions
\begin{displaymath}
  \zeta (\mathbf{\Delta}_Q)(s) = \tr^*(\mathbf{\Delta}_Q^{-s})=
  \frac{1}{\Gamma(s)}\int_0^{+\infty} \tr^* (
  e^{-t \mathbf{\Delta}_Q}) t^{s-1}dt \,,
\end{displaymath}
where $\tr^*$ denotes the trace over the non zero spectrum of
$\mathbf{\Delta}_Q$. These functions are well defined for
$\mathrm{Re}(s)$ large and meromorphic with (at worst) simple poles
occurring at
$s \in S = \{\frac{n+1 - j}{2} \mid j \in \N\} \setminus (-\N)$; see
e.g. \cite[Section~3.1]{Rumin-Seshadri} for references.  Following
\cite{Rumin-Seshadri}, we define then the contact torsion zeta
function
\begin{equation}
  \label{eq:2}
  Z(s) = \sum_{k=0}^n (-1)^k (n+1-k) \zeta (\mathbf{\Delta}_Q)(s)\,.
\end{equation}
Then, the analytic torsion of the contact complex is defined by
\begin{equation}
  \label{eq:3}
  T_Q (M,\rho) = \exp \bigl(-\frac{1}{2}Z'(0)\bigr) \,.
\end{equation}
It is shown in \cite{Rumin-Seshadri} that it coincides with Ray-Singer
analytic torsion on three dimensional CR Seifert manifolds. An
explicit formula is given in this case. On general contact manifolds,
Albin and Quan proved in \cite{Albin-Quan} that the Riemannian and
contact analytic torsions differ by integral of (unknown) local terms.

Our first main results relate this $\vartheta$ series to three other
expressions, one using topological data, another to an explicit
geometric sum, and the last one to dynamical properties of the Reeb
flow.

\subsection{The heat analytic torsion as an index series}
\label{sec:heat-analyt-tors}

\ \smallskip

We summarise the main steps toward the topological expression. As we
shall see, it turns out that the spectrum in the combination of trace
in $\vartheta$ is highly symmetric. Much of the plus and minus
contributions cancel each other out, except on a simple residual
spectrum we describe.

Let $\Omega^*H$ denotes the bundle of horizontal forms on $M$ with
coefficients in $V$, the flat bundle associated to the representation
$\rho : \pi_1(M) \rightarrow U(d)$ and consider the horizontal part of
the differential, $d_H$, acting on $\Omega^*H$ and the operator
\begin{displaymath}
  D_H = d_H + \delta_H\,.
\end{displaymath}
It exchanges $\Omega^{ev}H$ and $\Omega^{odd}H$. Let
$\mathcal{H}= \ker D_H$. We will show that
\begin{displaymath}
  \vartheta(t)= \tr(e^{t T^2} \mid \mathcal{H}^{ev}) - \tr(e^{t T^2} \mid
  \mathcal{H}^{odd} ).
\end{displaymath}

As we shall see in Section~\ref{sec:from-spectr-topol}, the space
$\mathcal{H}$ is infinite dimensional, and contains forms build using
CR (holomorphic) functions and their conjugate. We can however split
it into finite dimensional peaces using the spectrum of the Reeb flow
$T$. Fourier decomposition along the circle orbits, gives
\begin{equation}
  \label{eq:4}
  V = \bigoplus_{\lambda \in \spec(iT)} V_\lambda \,,
\end{equation}
where each component $V_\lambda$ can be seen as a $V$-bundle over the
orbifold $N = M / \mathbb{S}^1$. Then the spectral $\vartheta$ finally
reduces to a renormalised index series
\begin{equation}
  \label{eq:5}
  \vartheta(t) = \vartheta^{top}(t)= \sum_{\lambda \in \spec (iT)}
  \ind(D_H^{ev} \mid V_\lambda) 
  e^{-t\lambda^2} \,,
\end{equation}
with $D_H^{ev}=D_H: \Omega^{ev}H\rightarrow \Omega^{odd}H$. The index
terms can be explicitly computed using Kawasaki's index formula for
$V$-manifold; see Section~\ref{sec:index-computations}. This will link
$\vartheta$ to two other expressions, one using explicit geometric
data over $N$ and the other as a dynamical series over all closed
orbits.

\subsection{Geometric and dynamical viewpoints on the heat analytic
  torsion}
\label{sec:dynamical-viewpoint}

\ \smallskip

We now turn to the geometric and dynamical aspects of the series
$\vartheta$. We first would like to express it using data over the
orbifold $N$. In our case, its smooth part corresponds to the generic
closed primitive orbits $f$, and its finite set of singular conical
points are associated to exceptional orbits $f_i$ of order $\alpha_i$.

Recall that from the unitary representation
$\rho : \pi_1(M) \rightarrow U(d)$, each $\gamma \in \pi_1(M)$ induces
an holonomy map $\rho(\gamma)$ and a character value
\begin{displaymath}
  \chi_\rho (\gamma) = \tr (\rho(\gamma))\,.
\end{displaymath}
For any closed orbit we shall also need its algebraic length
$\dsp \ell(\gamma) = \int_\gamma \theta$. Note that since
$i_Td\theta= 0$, this length is always constant through orbit
deformation in contact geometry.

At last, the geometric expression will also relies on some weight on
the character. For $\gamma = f$ or $f_i$, we define
$V^x_\gamma = \ker (\rho(\gamma) -e^{2i\pi x} \id)$ and
\begin{equation}
  \label{eq:6}
  \chi^\theta_\rho(\gamma)(t) = \sum_{e^{2i \pi x} \in \spec{\rho(\gamma)}}
  \dim V^x_\gamma \,\theta(x, 4 \pi^2 t / \ell(\gamma)^2) 
\end{equation}
where
\begin{displaymath}
  \theta(x, t ) = \sum_{n\in \Z} e^{-t (n+x)^2}
\end{displaymath}
is a Jacobi theta function. As we shall see, this weighted character
is related to the average of the holonomies of closed random loops on
the circle $\gamma$.

\smallskip

Starting from the topological series \eqref{eq:5}, we will show the
following explicit geometric expression
\begin{equation}
  \label{eq:7}
  \vartheta = \vartheta^{geo}= \chi(N^*) \chi^\theta_\rho(f) + \sum_{i}
  \chi^\theta_\rho(f_i) \,, 
\end{equation}
where $\chi(N^*)$ is the Euler characteristic of the punctured
manifold $N^* = N \setminus \cup_i \{p_i \}$.

\smallskip

This will eventually lead to a Selberg-type trace formula giving an
expression of $\vartheta$ using data over the free homotopy classes
$\gamma$ of \emph{all} the closed orbits of the Reeb flow together
with their inverse, meaning the closed orbits of the reverse
flow. These are powers of the primitive closed orbits $f$ and
$f_i$. For these classes let
\begin{displaymath}
  e(\gamma) =
  \begin{cases}
    \ell(f) \chi(N) & \ \mathrm{if }\ \gamma \sim f^n \quad n \in \Z \,,\\
    \ell(f_i) & \ \mathrm{if }\ \gamma \sim f_i^n \quad n \not\equiv 0
    \mod \alpha_i\,,
  \end{cases}
\end{displaymath}
where $\dsp \chi(N) = \chi(N^*) + \sum_i \frac{1}{\alpha_i}$ is the
rational Euler class of the orbifold $N= M / \mathbb{S}^1$.

Generalising a result obtained in dimension $3$ in
\cite{Rumin-Seshadri}, we shall prove that
\begin{equation}
  \label{eq:8}
  \vartheta(t) = \vartheta^{dyn}(t) = \frac{1}{\sqrt{4\pi t}}
  \sum_\gamma \chi_\rho(\gamma) e(\gamma)
  e^{-\ell(\gamma)^2/ 4t}\,. 
\end{equation}
where the sum ranges over the homotopy classes of closed orbits of the
flow, together with their inverses and the constant loop.  All trace
formulae given here are invariant under the rescaling
$\theta \mapsto k \theta$ and $t\mapsto k^2 t$, which is specific to
the use of the contact de Rham complex instead of the usual Riemannian
one. It holds without particular assumption on the curvature or
symmetry of the Kähler orbifold $N$.

\smallskip

As we shall see in Section~\ref{sec:zeta-funct-viewp}, these
identities have counterparts using zeta type spectral functions
instead of heat ones.  They will lead in
Corollary~\ref{cor:lefsch-type-form-2} to explicit Lefschetz type and
dynamical formulae for the analytic contact torsion.

\smallskip

We now turn to the second spectral series we will be concerned in.

\subsection{Around the eta invariant in the contact setting.}
\label{sec:around-contact-eta}

\ \smallskip

In Riemannian geometry of dimension $4k-1$, the eta invariant is
defined using the odd signature operator
$$
S= (-1)^k(*d + d*) w
$$
acting (for convenience here) on odd forms, with $*$ the Hodge star
and $w=(-1)^p$ on $\Omega^{2p-1} M$; see \cite[p.~63]{At-Pa-SiI}. This
operator is self-adjoint and $S^2 = \Delta$ is Hodge--de Rham
Laplacian.  The eta invariant is given by the value at $s=0$ of the
meromorphic function
\begin{displaymath}
  \eta(S)(s)= \tr (S |S|^{-2s-1})= 
  \frac{1}{\Gamma(s+1/2)}
  \int_0^{+\infty} \tr( \sqrt  t S e^{-t\Delta}) t^{s-1} dt.
\end{displaymath}
As $S$ maps $\Omega^{2p-1}M $ to
$\Omega^{2(2k-p)-1}M \oplus \Omega^{2(2k-p+1)-1}M$, only forms in
'middle degree' $\Omega^{2k-1}M$ contribute to the trace in $\eta(S)$,
so that
\begin{displaymath}
  \eta (S)(s) =  \frac{1}{\Gamma(s+1/2)} \int_0^{+\infty} \tr (\sqrt t(*d)
  e^{-t\Delta} \mid \Omega^{2k-1}M) t^{s-1} dt\,.
\end{displaymath}

Now, it has been shown that $\eta(S)(0)$ is related to the eta
invariant of its contact second order counterpart $*D$ acting on
$\mathcal{E}^{2k-1}$; see \cite{BHR} for the three dimensional Seifert
case and Albin-Quan's more recent work \cite[\S 6]{Albin-Quan} on
general contact manifolds. The difference is given by the integral of
(unknown) universal curvature polynomial. Although it captures the eta
invariant, the operator $*D$ itself has not good analytic properties,
due to its infinite dimensional kernel. It is not hypoelliptic. It
needs to be completed by some extra term like in the signature
Riemannian operator $S$. Possible choices could be
$P= *D \pm d_Q\delta_Q$ on $\mathcal{E}^{2k-1}$. One has
$P^2 = \mathbf{\Delta}_Q$ and
$\eta (P) = \eta (*D) \pm \zeta(d_Q \delta_Q)$, where the zeta series
of the positive operator $d_Q \delta_Q$ contributes to an alternating
sum of cohomological dimensions up to some local term at $s=0$. In
higher dimension however, these choices of 'extensions' of $*D$ don't
seem to be the most natural ones in terms of spectral symmetry.

\smallskip

We will consider instead the operator defined by
\begin{equation}
  \label{eq:9}
  S_Q =
  \begin{cases}
    *D + (d_Q + \delta_Q) \sigma \delta_Q & \mathrm{on}\
    \mathcal{E}^{2k-1}\\
    (d_Q+\delta_Q)\sigma (d_Q+\delta_Q) & \mathrm{on} \
    \bigoplus_{1\leq p\leq k-1 } \mathcal{E}^{2k-1-2p}
  \end{cases}
\end{equation}
where $\sigma= (-1)^p J$ on $\mathcal{E}^{2p}$ and $J=i^{a-b}$ on
forms $\mathcal{E}^{a,b}$ of bidegre $(a,b)$ with respect to the
complex structure. We shall see in Proposition~\ref{prop:eta-etaS_Q}
that when the CR structure has a transverse symmetry,
$S_Q^2 = \mathbf{\Delta_Q}$ and still
\begin{displaymath}
  \eta(S_Q)= \eta(*D) + \sum \pm \ \zeta(\Delta_Q)
\end{displaymath}
leading again to adding cohomological dimensions and local terms at
$s=0$ . The advantage of this choice of signature operator lies in its
extra symmetry with respect to $\sigma= (-1)^p J$ on
$\mathcal{E}^{2p-1}$. It splits into
\begin{displaymath}
  S_Q = \sigma T + P
\end{displaymath}
here $\sigma P = - P \sigma$ while $\sigma T = T \sigma$ and
$T P = P T$, so that the spectrum of $S_Q$ is symmetric except on (an
infinite dimensional space) $\ker P = \mathcal{H}_S$ on which
$S_Q = \sigma T$.

\subsection{The contact eta trace as topological and dynamical
  series.}
\label{sec:contact-eta-trace}

\ \smallskip

As in the previous case of the analytic torsion, the spectral series
involved in $\eta(S_Q)$ have both closed topological and dynamical
expression. Let
\begin{equation}
  \label{eq:10}
  \vartheta_S (t) = \tr (\sqrt t S_Q e^{-t\mathbf{\Delta}_Q})\,.
\end{equation}
The domain of $S_Q$
\begin{displaymath}
  \mathcal{E}_S = \bigoplus_{1\leq p\leq k} \mathcal{E}^{2p -1}
\end{displaymath}
splits into $\mathcal{E}_S^+ \oplus \mathcal{E}_S^-$ with respect to
the involution $\tau= i\sigma$, as does
$\mathcal{H}_S = \mathcal{H}_S^+ \oplus \mathcal{H}_S^-$. The operator
$P$ exchanges this splitting and we set
$P^+ = P : \mathcal{E}^+_S \rightarrow \mathcal{E}^-_S$. By the
previous discussion $\vartheta_S(t)$ reduces on $\mathcal{H}_S$ and we
have that
\begin{displaymath}
  \vartheta_S(t)
  = - \tr(i\sqrt t Te^{tT^2 }\mid \mathcal{H}_S^+) + \tr(i\sqrt t T
  e^{tT^2} \mid \mathcal{H}_S^-)\,.
\end{displaymath}
Using the same splitting of $V$ through the circle action as
in~\eqref{eq:4}, we will finally get
\begin{displaymath}
  \vartheta_S(t)
  = \vartheta^{top}_S (t)= - \sqrt t \sum_{\lambda \in \spec(iT)} \ind( P^+
  \mid V_\lambda) \lambda e^{-t\lambda^2} \,,
\end{displaymath}
where, as we shall see in Section~\ref{sec:tors-cont-trace}, the index
there is the signature of the (non flat) bundle $V_\lambda$ over the
orbifold $N = M / \mathbb{S}^1$. This is the first interpretation of
$\vartheta_S$ as an index series. This leads to an explicit geometric
expression using Kawasaki's index formula.

\smallskip

We now turn to the link with dynamical data. The objective is to
single out the contribution of each closed orbit like in the analytic
torsion case \eqref{eq:8}. Let
\begin{displaymath}
  \c= c_1(L)= -\dfrac{d\theta}{2\pi}
\end{displaymath}
be the first Chern class of $M$ seen as the circle bundle of a complex
line bundle over $N=M/ \mathbf{S}^1$. Let $\mathcal{L}(N)$ be
Hirzebruch $L$-genus of the (smooth part of the) orbit space $N$.

We shall see in Section~\ref{sec:from-topol-geom} that
\begin{displaymath}
  \vartheta_S(t) = \vartheta_S^{dyn}(t) = \frac{1}{\sqrt {4 \pi}}
  \sum_\gamma \chi_\rho(\gamma) \sigma(\gamma)(t),
\end{displaymath}
where powers of the generic orbit $\gamma= f^n$ contribute to
\begin{displaymath}
  \sigma(\gamma)(t) =
  \frac{i \ell(f)}{2t} \, \bigl\langle (\ell(\gamma) + i \c) e^{-(\ell
    (\gamma) + i \c)^2 /4t} \wedge \mathcal{L}(N), [N_{smooth}]
  \bigr\rangle \,, 
\end{displaymath}
while powers of the singular orbits $\gamma= f_i^k$ with
$k \not\equiv 0 \mod\alpha_i$ contribute to
\begin{displaymath}
  \sigma(\gamma)(t) = \frac{i \ell(f_i)}{2t} \ell(\gamma) e^{-\ell(\gamma)^2/4t}
  \, \nu(\gamma)  
\end{displaymath}
with
\begin{displaymath}
  \nu(\gamma) = i(-1)^k \prod_{j=1}^{2k-1} \cot ( \theta_j/2) \,.
\end{displaymath}
Here $\theta_j$ are the angles of the action of $\gamma$ on the
horizontal space $H$. Following Atiyah-Bott's work \cite{Atiyah-Bott},
$\nu(\gamma)$ arises in the Lefschetz fixed point formula for the
signature operator.

\smallskip

This will eventually lead to explicit expressions of the twisted eta
invariant $\eta(S_Q)(0)$ in terms of these topological and dynamical
data; see Sections~\ref{sec:appl-cont-eta-1} and
\ref{sec:contact-eta-function}.

\section{Review of basic constructions and miscellaneous formulae}
\label{sec:revi-misc-form}

To make the paper as self contained as possible we will start by
discussing the contact de Rham complex because it plays an important
role here. We will also review miscellaneous formulae around it. Much
of this material can be found in other places; see
e.g. \cite{Rumin94,Rumin00, Rumin-Seshadri}

Let $M$ be a smooth manifold of dimension $2n+1$. A $2n$-dimensional
sub-bundle $H \subset TM$ is a \emph{contact distribution} if a
$1$-form $\theta$ such that $H = \ker \theta$ satisfies the non
integrability condition $\theta \wedge d\theta^n \not=0$. Such a form
is called a \emph{contact form}. Associated to a choice of $\theta$ is
the transverse \emph{Reeb field} $T$; it is the unique vector field
satisfying $\theta(T)=1$ and $\mathcal{L}_T \theta = i_T d \theta= 0$,
where $\mathcal{L}_T$ is Lie derivative along $T$.

The exterior algebra of $M$ splits into horizontal and vertical forms
\begin{displaymath}
  \Omega^*M = \Omega^* H \oplus \theta \wedge \Omega^* H
\end{displaymath}
where $\Omega^* H $ are forms vanishing on $T$. The exterior
differential $d$ on $\Omega^* M$ writes then
\begin{displaymath}
  d (\alpha_H + \theta \wedge \alpha_T) = (d_H \alpha_H + d \theta\wedge
  \alpha_T) + 
  \theta \wedge (T \alpha_H - d_H \alpha_T)
\end{displaymath}
using the notation $T=\mathcal{L}_T$ on forms, that is in matrix form
\begin{displaymath}
  d =
  \begin{pmatrix}
    d_H & L\\
    T & - d_H
  \end{pmatrix},
\end{displaymath}
where $d_H=\Pi_{\Omega^*H}d$ is the horizontal part of $d$ (that skips
the differential along $T$), and $L \alpha = d\theta \wedge
\alpha$. From $d^2= 0$, one gets
\begin{equation}
  \label{eq:11}
  d_H^2 = - L T\ ,\quad [L, T]= 0 = [L, d_H]\,. 
\end{equation}
Note that $d_H$ is not a complex, and moreover that the splitting of
$\Omega^*M$ and $d_H$ depends on the choice of a contact form
$\theta$. It is possible to construct another sequence of operators
that avoid this.

\subsection{The contact complex.}
\label{sec:contact-complex}

\ \smallskip

Let $\mathcal{I}^* $ be the ideal in $\Omega^*M$ generated by $\theta$
and $d\theta$
$$
\mathcal{I}^*= \{ \alpha \in \Omega^* M \mid \alpha = \theta \wedge
\beta + d \theta \wedge \gamma\}\,,
$$
and $\mathcal{J}^*$ its annihilator
$$
\mathcal{J}^*=\{ \alpha \in \Omega^*M \mid \theta \wedge \alpha = d
\theta \wedge \alpha = 0\}\,.
$$
They are independent on the choice of contact form and stable under
$d$. From e.g. \cite{Weil}, $L$ is injective on $\Omega^k H$ for
$k \leq n-1$ and surjective onto $\Omega^k H$ for $k\geq n+1$. Hence
$\mathcal{I}^k = \Omega^k M$ for $k \geq n+1$ and $\mathcal{J}^k = 0$
for $k \leq n+1$. Then de Rham exterior differential induces two
Quotiented complexes
\begin{displaymath}
  \Omega^0M\stackrel{d_Q}{\longrightarrow}\Omega^1M/
  \mathcal{I}^1\stackrel{d_Q}{\longrightarrow}
  \cdots\stackrel{d_Q}{\longrightarrow}\Omega^nM/\mathcal{I}^n 
\end{displaymath}
and
\begin{displaymath}
  \mathcal{J}^{n+1}\stackrel{d_Q}{\longrightarrow}
  \mathcal{J}^{n+2}\stackrel{d_Q}{\longrightarrow}
  \cdots\stackrel{d_Q}{\longrightarrow}\mathcal{J}^{2n+1}.
\end{displaymath}
These can be joined using the following:
\begin{lemma}\cite[p.~286]{Rumin94}
  Let $\alpha \in \Omega^n M / \{ \mathrm{vertical\ forms}\}$. Then
  there exists a unique lift $\overline \alpha $ of $\alpha$ in
  $\Omega^n M$ such that $d \overline \alpha \in
  \mathcal{J}^{n+1}$. Moreover $d \overline \alpha= 0$ if
  $\alpha = d \theta \wedge \beta$.
\end{lemma}
One defines then
$D : \Omega^nM/\mathcal{I}^n \rightarrow \mathcal{J}^{n+1}$ by
$D\alpha = d \overline \alpha$. Note that $D$ is a \emph{second order}
operator, taking $T$ as a second order one in our contact setting by
\eqref{eq:11}. Given a choice of contact form, the formula for $D$
reads
\begin{equation}
  \label{eq:12}
  D \alpha = d (\alpha_H - \theta \wedge L^{-1}d_H \alpha_H) = \theta
  \wedge (T +  d_H L^{-1}d_H ) \alpha_H \,,
\end{equation}
if $\alpha_H $ is the representative of $\alpha$ in $\Omega^n H$.  The
so-called contact complex is then
\begin{displaymath}
  \Omega^0M\stackrel{d_Q}{\longrightarrow}\Omega^1M/
  \mathcal{I}^1\stackrel{d_Q}{\longrightarrow}
  \cdots\stackrel{d_Q}{\longrightarrow}\Omega^nM/\mathcal{I}^n 
  \stackrel{D}{\longrightarrow}
  \mathcal{J}^{n+1}\stackrel{d_Q}{\longrightarrow}
  \mathcal{J}^{n+2}\stackrel{d_Q}{\longrightarrow}
  \cdots\stackrel{d_Q}{\longrightarrow}\mathcal{J}^{2n+1}.
\end{displaymath}
We have :
\begin{prop}\cite[p.~286]{Rumin94}
  \label{prop:contact-complex-1}
  The contact complex is a resolution of the constant sheaf and hence
  its cohomology coincides with de Rham cohomology of $M$. Moreover
  the canonical projections
  $\pi : \Omega^k M \rightarrow \Omega^k M / \mathcal{I}^k$ for
  $k \leq n$ and injections $i : \mathcal{J}^k \rightarrow \Omega^k M$
  for $k\geq n+1$ induce isomorphism between the two cohomologies.
\end{prop}
The arguments being purely local, these results also apply on twisted
version of the complexes with a flat bundle $V$, as coming from a
representation $\rho : \pi_1(M) \rightarrow U(d)$.

\medskip

Using a complex structure $J$ on $H$ such that
$d \theta(\cdot, J\cdot) $ is Hermitian positive definite, one defines
a Riemannian metric on $M$
\begin{displaymath}
  g = d \theta (\cdot, J \cdot) + \theta^2\,.
\end{displaymath}
Let then $\Lambda = L^* $ be the adjoint of
$L : \Omega^k H \rightarrow \Omega^{k+2 } H$ where
$L\alpha= d\theta \wedge \alpha$, and $\Omega^*_0H = \ker \Lambda$ be
the bundle of primitive horizontal forms.  We will identify in the
sequel the quotient spaces $\Omega^k M / \mathcal{I}^k$ in the
lower-half of the contact complex with $\Omega^k_0 H$. Let
\begin{equation}
  \label{eq:13}
  \mathcal{E}^k =
  \begin{cases}
    \Omega^k_0  H & \mathrm{if}\ k \leq n\\
    \mathcal{J}^k & \mathrm{if}\ k \geq n+1.
  \end{cases}
\end{equation}
be the definition spaces of the contact complex in this
identification. Note that Hodge star operator $*$ exchanges
$\mathcal{E}^k$ and $\mathcal{E}^{2n+1-k}$.

\subsection{Miscellaneous formulae.}
\label{sec:misc-form}

\ \smallskip

We gather now some useful identities; see e.g. \cite[Section
4]{Rumin00} for more details. The first ones are similar to basic
formulae from Kählerian geometry, see \cite{Weil}. At the algebraic
level, it holds on the Hermitian space $H$ that
\begin{equation}
  \label{eq:14}
  [\Lambda, L] = n-p \quad \mathrm{on} \ \Omega^pH\,.
\end{equation}
Moreover, following \cite[Thm.~3]{Weil} for instance, $\Omega^*H$
splits under the Lefschetz decomposition
\begin{equation}
  \label{eq:15}
  \Omega^* H = \bigoplus_{0\leq k \leq q \leq n} L^k \Omega^{n-q}_0H =
  \bigoplus_{0\leq k \leq q \leq n} L^k \mathcal{E}^{n-q}   
  \,.
\end{equation}

At the level of first order operators one has
\begin{equation}
  \label{eq:16}
  [\Lambda, d_H] = -\delta_H^J \,.
\end{equation}
where $\delta_H^J = J^{-1} \delta_H J$ and
$J\alpha (X_1, \cdots, X_p) = \alpha(J X_1, \cdots, J X_p)$ on
$\Omega^p H$.

This leads to the action of $d_H$ with respect to Lefschetz
decomposition. Thanks to \eqref{eq:14} and \eqref{eq:16}, it holds on
$\Omega^p_0H$ that
\begin{equation}
  \label{eq:17}
  d_H = d_Q - \frac{L}{n-p+1}\delta_Q^J\,,
\end{equation}
with the convention here that $d_Q=0$ on $\Omega^n_0H$. This extends
on $L^k\Omega^p_0H$ using $[L,d_H]=0$ by \eqref{eq:11}.

We now come to second order relations on the contact complex. From
\eqref{eq:11} \eqref{eq:12} and \eqref{eq:17}, one gets:
\begin{equation}
  \label{eq:18}
  T =
  \begin{cases}
    \frac{1}{n-p} \delta_Q^Jd_Q + \frac{1}{n-p+1} d_Q \delta_Q^J & \
    \mathrm{on}\ \mathcal{E}^p \ \ \mathrm{for}\ p < n\\
    i_T D + d_Q \delta_Q^J & \ \mathrm{on} \ \mathcal{E}^n \,.
  \end{cases}
\end{equation}
In order to get rid of the multiplicative coefficients in formulae as
above, we will normalise the differentials $d_Q$ as in
\cite[p.~418]{Rumin00}. Namely on $\mathcal{E}^p$ for $p < n$, we
shall use from now on
\begin{equation}
  \label{eq:19}
  \frac{1}{\sqrt{n-p} } d_Q \quad \mathrm{instead \ of }\quad d_Q\,.
\end{equation}
We will keep the same notation $d_Q$ for this normalised differential
in the sequel since we will only use them. Hence \eqref{eq:18} reads
now
\begin{equation}
  \label{eq:20}
  T =
  \begin{cases}
    \delta_Q^Jd_Q + d_Q \delta_Q^J & \
    \mathrm{on}\ \mathcal{E}^p \ \ \mathrm{for}\ p < n\\
    i_T D + d_Q \delta_Q^J & \ \mathrm{on} \ \mathcal{E}^n \,.
  \end{cases}
\end{equation}

Using this and $L = L^J$, $d\theta $ being a $(1,1)$ form, one deduces
that on $\Omega^*H$
\begin{equation}
  \label{eq:21}
  T^* = - T^J \quad \mathrm{and}\quad
  [\Lambda,T]= 0\,.
\end{equation}

\smallskip

So far, all identities here hold for any calibrated complex structure
$J$, i.e. satisfying that $d\theta(\cdot, J \cdot)$ is positive
Hermitian.  In the sequel, we will assume moreover that $J$ is
\emph{integrable}, meaning that $[H^{1,0}, H^{1,0}] \subset
H^{1,0}$. In that case both $d_H$ and $d_Q$ split into two components
\begin{equation}
  \label{eq:22}
  d_H = \partial_H + \overline\partial_H \quad \mathrm{and}\quad d_Q =
  \partial_Q + \overline\partial_Q\,,
\end{equation}
where $\partial_{H,Q}$ increases the bidegree by $(1,0)$ and
$\overline\partial_{H,Q}$ by $(0,1)$. Developing $d_Q^2=0$ on
$\mathcal{E}^p$ for $p\leq n-2$, first gives
\begin{equation}
  \label{eq:23}
  \partial_Q^2 = \overline\partial_Q^2 = 0= \partial_Q
  \overline\partial_Q + \overline\partial_Q \partial_Q = d_Q d_Q^J +
  d_Q^J d_Q\,.
\end{equation}

We can also get other second order relations between the
$Q$-differentials by developing \eqref{eq:20} on $\mathcal{E}^p$ with
$p< n$.  This gives
\begin{equation}
  \label{eq:24}
  \left\{
    \begin{aligned}
      \Delta_{\overline\partial_Q } - \Delta_{\partial_Q}
      & = i T^{0,0}\\
      \overline\partial_Q^* \partial_Q + \partial_Q
      \overline\partial_Q^* & = i T^{1,-1}\\
      \partial_Q^*\overline\partial_Q + \overline\partial_Q
      \partial_Q^* & = -i T^{-1,1}\\
    \end{aligned}
  \right.
\end{equation}
where
\begin{displaymath}
  \Delta_{\overline\partial_Q} = \overline\partial_Q^*
  \overline\partial_Q + \overline\partial_Q \overline\partial_Q^*
  \quad \mathrm{and}\quad
  \Delta_{\partial_Q} = \partial_Q^* \partial_Q + \partial_Q \partial_Q^*\,.
\end{displaymath}
At this point we note that
\begin{displaymath}
  T - T^J = - J^{-1}(\mathcal{L}_TJ) = (1+i)
  T^{1,-1}+ (1-i) T^{-1,1} 
\end{displaymath}
is a zero order algebraic operator that vanishes when the Reeb flow
preserves the complex structure, thus the metric. In conclusion we
have the following:
\begin{prop}\cite[p.~418]{Rumin00}
  \label{prop:Laplacian-zero-torsion}
  Suppose that the complex structure on a CR contact manifold is
  integrable and preserved by the Reeb flow.

  Then it holds on $\mathcal{E}^p$ for $p< n$ that the second order
  $Q$-Laplacian $ \Delta_Q= d_Q \delta_Q + \delta_Q d_Q$ commutes with
  $J$ and writes
  \begin{equation}
    \label{eq:25}
    \Delta_Q = \Delta_{\partial_Q}+ \Delta_{\overline\partial_Q} 
    \quad \mathrm{with}\quad 
    \Delta_{\overline\partial_Q } - \Delta_{\partial_Q}
    = i T\,.
  \end{equation}
\end{prop}
\begin{rem}
  We recall that these results hold with the renormalised
  differentials as defined in \eqref{eq:19}.
\end{rem}

\subsection{The middle degree case.}
\label{sec:middle-degree-case}

\ \smallskip

At this point we still miss a $Q$-Laplacian in middle degree. The
differential $D : \mathcal{E}^n\rightarrow \mathcal{E}^{n+1}$ here is
second order. Hence a starting expression for a positive Laplacian is
the fourth order $D^* D$. A natural way to complete it, is to set on
$\mathcal{E}^n$
\begin{equation}
  \label{eq:26}
  \mathbf{\Delta}_Q = (d_Q\delta_Q)^2 + D^*D \,.
\end{equation}
An important feature of this choice lies in its commuting property
with $J$, as in the lower degree case.
\begin{prop}\cite[Prop.~4]{Rumin94}
  \label{prop:middle-degree-case-1}
  $\mathbf{\Delta_Q}$ preserves the bidegree of forms in
  $\mathcal{E}^n$ when the complex structure is integrable and
  invariant through the Reeb flow.
\end{prop}
\begin{proof}
  From \eqref{eq:21} one has $T^* = -T$ when $\mathcal{L}_TJ=0$, and
  then by \eqref{eq:18}
  \begin{align*}
    D^*D & = (i_T D)^* (i_TD) = (-T -d_Q^J
           \delta_Q)(T - d_Q \delta_Q^J)  \\
         & = -T^2 + T (-d_Q^J \delta_Q +d_Q
           \delta_Q^J) + d_Q^J\delta_Q d_Q \delta_Q^J \,.
  \end{align*}
  There $T$ and $-d_Q^J \delta_Q + d_Q \delta_Q^J$ commute with $J$,
  whereas by Proposition~\ref{prop:Laplacian-zero-torsion}
  \begin{align*}
    d_Q^J \delta_Q d_Q \delta_Q^J
    & = d_Q^J (\Delta_Q -
      d_Q \delta_Q) \delta_Q^J = d_Q^J \Delta_Q^J \delta_Q^J - d_Q^J d_Q
      \delta_Q \delta_Q^J \\
    & = (d_Q^J \delta_Q^J)^2 - d_Q^J d_Q \delta_Q \delta_Q^J\,.
  \end{align*}
  The last term preserves the bidgree by \eqref{eq:23}. Finally adding
  $(d_Q \delta_Q)^2$ gives that $\mathbf{\Delta}_Q$ commutes with
  $J$. Note that $\mathbf{\Delta}_Q$ has no $(2,-2)$ (and $(-2,2)$)
  component neither since it can come only from combination of type
  \begin{displaymath}
    \partial_Q \overline\partial_Q^* \partial_Q \overline\partial_Q^*
    = -\partial_Q \partial_Q \overline\partial_Q^*
    \overline\partial_Q^* =0 \,,
  \end{displaymath}
  by \eqref{eq:24} and \eqref{eq:23}.
\end{proof}

In order to get fourth order Laplacians for all degrees we shall
define
\begin{equation}
  \label{eq:27}
  \mathbf{\Delta}_Q = \Delta_Q^2 = (d_Q\delta_Q)^2 + (\delta_Q
  d_Q)^2 \quad \mathrm{on} \ \mathcal{E}^p \
  \mathrm{for}\ p < n\,.
\end{equation}
We extend these operators on all $\Omega^*H$ using the Lefschetz
decomposition \eqref{eq:15} and requiring that
\begin{displaymath}
  L \mathbf{\Delta}_Q =
  \mathbf{\Delta}_Q L \,.
\end{displaymath}
In this way, the Laplacian $\mathbf{\Delta}_Q$ commutes with all the
algebra of operators we face here: $L$, $J$, $d_Q$, $d_H$ and their
adjoints. It will play the role of a ``Casimir'' operator in our
situation.

\smallskip

To conclude this part, we mention the main analytic result on the
$Q$-Laplacian.
\begin{thm}\cite[p.~290]{Rumin94}
  The $Q$-Laplacians $\Delta_Q$ and $\mathbf{\Delta}_Q$ are maximally
  hypoelliptic on any compact contact manifold.
\end{thm}
Roughly speaking, that means that these operators control as much
horizontal derivatives as possible: two for $\Delta_Q$ and four for
$\mathbf{\Delta}_Q$. See for instance the discussion in
\cite[Section~3.1]{Rumin-Seshadri} for a presentation of main
properties and references about this analytic notion. A consequence
for this paper is that on compact contact manifolds, the
$Q$-Laplacians are self-adjoint and possess pure point
spectrum. Moreover the associated heat kernels
$e^{-t\mathbf{\Delta}_Q}$ are smooth hence trace class in this
setting.

\section{Results on the contact analytic torsion}
\label{sec:results-cont-analyt}

We now state our main result on the torsion function.  We first define
the contact manifolds we will be concerned in.

\begin{defn}
  \label{def:Seifert_CR}
  Let $M$ be a compact contact manifold of any dimension. We shall say
  that $M$ is a \emph{CR Seifert manifold} if it admits a locally free
  circle action generated by a Reeb field $T$ that preserves moreover
  an integrable complex (CR) structure $J$ on $H$.
\end{defn}
The orbit space $N = M/\mathbf{S^1}$ of such a manifold inherits a
structure of a K\"ahlerian orbifold with cyclic quotient
singularities; see e.g. \cite[p.~52]{Fried87}.  Note that these
manifold are particular cases of Sasakian manifolds, called
quasi-regular, whose Reeb field have closed orbits; see
\cite{Ornea-Verbitsky}.

We also endow $M$ with a unitary representation
$\rho: \pi_1(M) \rightarrow U(d)$ (that can be trivial). This possibly
allows to twist the contact complex by taking values in the associated
flat bundle $V$. That will highlight the role of each individual
homotopy class of closed orbit in the dynamical expression.

\smallskip

Using notations introduced in Sections~\ref{sec:heat-analyt-tors} and
\ref{sec:dynamical-viewpoint} we recall that the initial heat torsion
function is spectral and given by \eqref{eq:1}
\begin{displaymath}
  \vartheta(t) = \sum_{k=0}^n (-1)^k (n+1-k) \tr (e^{-t\mathbf{\Delta}_Q}\mid
  \mathcal{E}^k). 
\end{displaymath}
We also define the topological theta function
\begin{displaymath}
  \vartheta^{top} (t) = \sum_{\lambda \in \spec (iT)}
  \ind(D_H^{ev} \mid V_\lambda)  e^{-t\lambda^2} \,,
\end{displaymath}
the geometric expression
\begin{displaymath}
  \vartheta^{geo} = \chi(N^*) \chi^\theta_\rho(f) + \sum_{i}
  \chi^\theta_\rho(f_i) \,, 
\end{displaymath}
and the dynamical series
\begin{displaymath}
  \vartheta^{dyn}(t) = \frac{1}{\sqrt{4\pi t}}
  \sum_\gamma \chi_\rho(\gamma) e(\gamma) e^{-\ell(\gamma)^2/ 4t}\,. 
\end{displaymath}
We shall prove the following.
\begin{thm}
  \label{thm:torsion_spec_top_dyn}
  Let $M$ be a CR Seifert manifold of any dimension endowed with a
  unitary representation. Then it holds that
  \begin{displaymath}
    \vartheta = \vartheta^{top}= \vartheta^{geo} = \vartheta^{dyn} \,.
  \end{displaymath}
\end{thm}

\subsection{From spectral to topological torsion series.}
\label{sec:from-spectr-topol}

\ \smallskip

We start with the proof of the first identity
$\vartheta= \vartheta^{top}$. Recall that by Lefschetz decomposition
\eqref{eq:15}
\begin{displaymath}
  \Omega^* H =
  \bigoplus_{0\leq p \leq q \leq n} L^p \mathcal{E}^{n-q} \,.
\end{displaymath}
Therefore $\Omega^*H$ contains $n+1-k$ copies of each $\mathcal{E}^k$
for $0\leq k\leq n$. This is the multiplicity of $\mathcal{E}^k$ in
$\vartheta$. Moreover, since $L= d\theta\wedge\cdot$ preserves the
parity, the parity of forms in $\Omega^*H$ and their components in
$\mathcal{E}^*$ coincides. Thus the plus and minus parts of
$\vartheta$ combine to give
\begin{displaymath}
  \vartheta(t) = \tr (e^{-t \mathbf{\Delta}_Q} \mid \Omega^{ev} H) -
  \tr (e^{-t \mathbf{\Delta}_Q} 
  \mid \Omega^{odd} H) \,. 
\end{displaymath}

Now $D_H = d_H + \delta_H$ exchanges $\Omega^{ev}H$ and
$\Omega^{odd}H$ and commutes with $\mathbf{\Delta}_Q$ by
Section~\ref{prop:middle-degree-case-1}. Hence, setting
$\mathcal{H}= \ker D_H$, one has
\begin{displaymath}
  \vartheta(t) = \tr (e^{-t \mathbf{\Delta}_Q} \mid \mathcal{H}^{ev}) -
  \tr (e^{-t \mathbf{\Delta}_Q}  \mid \mathcal{H}^{odd}) \,.
\end{displaymath}
We need the following Lemma to identify the residual spectrum on
$\mathcal{H}$.
\begin{lemma}
  \label{lemma:Delta_Q_on_H}
  On CR Seifert manifolds, one has $\mathbf{\Delta}_Q = -T^2$ on
  $\mathcal{H}=\ker D_H$.
\end{lemma}
Note that $\mathcal{H}=\ker D_H$ should not be confused with
$\ker \Delta_H=\ker d_H \cap \ker \delta_H$ since $d_H$ is not a
complex by \eqref{eq:11}. The next proposition will lead us to a
useful description of $\mathcal{H}$.
\begin{prop}
  \label{prop:conjugation_D_H}
  Let $P=L+\Lambda$ acting on $\Omega^*H$ and $U= e^{i\pi P/4}$. Then
  on a contact manifold with an integrable $J$, it holds that
  \begin{displaymath}
    \left\{
      \begin{aligned}
        [P,D_H]& = d_H^J - \delta_H^J \\
        [P,d_H^J -\delta_H^J] & = D_H
      \end{aligned}
    \right.
  \end{displaymath}
  and
  \begin{displaymath}
    U^{-1}D_H U = \sqrt2 (\overline\partial_H + \overline\partial_H^*)\,.
  \end{displaymath}
\end{prop}
\begin{proof}
  The first identities come from \eqref{eq:11} and \eqref{eq:16}
  giving $[L,d_H]= 0$ and $[\Lambda,d_H]=-\delta_H^J$ and conjugated
  relations. Therefore for any angle $\varphi$
  \begin{align*}
    e^{i\varphi\, \mathrm{ad}(P)} D_H
    &= \cos \varphi D_H + i \sin
      \varphi (d_H^J-\delta_H^J) \\
    & = \mathrm{Ad}(e^{i\varphi P}) D_H = e^{i\varphi P}  D_H
      e^{-i\varphi P} \,.
  \end{align*}
  This yields the last identity using $\varphi= -\pi/4$ and the
  splitting of $d_H = \partial_H + \overline\partial_H$ for integrable
  $J$ by \eqref{eq:22}.
\end{proof}

Let
$\mathcal{H}_{\overline\partial_H}=\ker (\overline\partial_H +
\overline\partial_H^*)$.  By the previous proposition
\begin{displaymath}
  \mathcal{H}= \ker D_H = U (\mathcal{H}_{\overline\partial_H})\,,
\end{displaymath}
with the advantage that on CR Seifert manifolds
$\overline\partial_H^2=0$ from \eqref{eq:11}. Hence one has also
$\mathcal{H}_{\overline\partial_H}= \ker \Delta_{\overline\partial_H}$
with
\begin{displaymath}
  \Delta_{\overline\partial_H} = \overline\partial_H \overline\partial_H^*
  + \overline\partial_H^* \overline\partial_H \,.
\end{displaymath} 
Now by definition $\Delta_{\overline\partial_H}$ preserves the
bi-degree of forms, but also the spaces $L^k \Omega^p_0H$ in the
Lefschetz decomposition. This is due to the commutation relation
\begin{displaymath}
  \Delta_{\overline\partial_H}L = L(\Delta_{\overline\partial_H} -i T)
\end{displaymath}
that comes from \eqref{eq:16} when $\mathcal{L}_T J=0$. Indeed, from
\begin{displaymath}
  [L, \overline\partial_H] = 0 \quad \mathrm{and} \quad [L,
  \overline\partial_H^*]= -i \partial_H  
\end{displaymath}
one gets
\begin{displaymath}
  \Delta_{\overline\partial_H}L - L\Delta_{\overline\partial_H} = i
  (\partial_H \overline \partial_H + \overline\partial_H \partial_H) = -iL
  T 
\end{displaymath}
by \eqref{eq:11}. This leads eventually to the splitting of
$\mathcal{H}_{\overline\partial_H}$ into forms of pure bidegree and
type in the Lefschetz decomposition. Using \eqref{eq:17} we first
find.
\begin{prop}
  \label{prop:decomposition-kerH-dbar}
  On a CR Seifert manifold, one has
  \begin{displaymath}
    \ker \overline\partial_H =
    \begin{cases}
      \ker \overline\partial_Q \cap \ker \partial_Q^* & \quad
      \mathrm{on}\
      L^k \mathcal{E}^{n-q} \ \mathrm{for}\ 0\leq  k< q\leq n\\
      \ker \partial_Q^* & \quad \mathrm{on} \ L^q \mathcal{E}^{n-q}
    \end{cases}
  \end{displaymath}
  and
  \begin{displaymath}
    \ker \overline\partial_H^*  =
    \begin{cases}
      \ker \overline\partial_Q^* \cap \ker \partial_Q & \quad
      \mathrm{on}\
      L^k \mathcal{E}^{n-q} \ \mathrm{for}\ 1 \leq k\leq q \leq n\\
      \ker \overline\partial_Q^* & \quad \mathrm{on} \ \mathcal{E}^{p}
      \ \mathrm{for} \ 0\leq p\leq n\,.
    \end{cases}
  \end{displaymath}
\end{prop}
In conclusion
\begin{displaymath}
  \mathcal{H}_{\overline\partial_H} =
  \begin{cases}
    \ker \Delta_Q & \quad \mathrm{on} \ L^k\mathcal{E}^{n-q} \
    \mathrm{for}\
    1\leq k< q \leq n\\
    \ker \overline\partial_Q \cap \ker \partial_Q^* \cap \ker
    \overline\partial_Q^* & \quad \mathrm{on}\ \mathcal{E}^p \
    \mathrm{for}\ p< n\\
    \ker \partial_Q^* \cap \ker \overline\partial_Q^* \cap \ker
    \partial_Q &
    \quad \mathrm{on} \ L^q \mathcal{E}^{n-q}\ \mathrm{for}\ 0 < q \leq n\\
    \ker \partial_Q^* \cap \ker \overline\partial_Q^* & \quad
    \mathrm{on} \ \mathcal{E}^n \,.
  \end{cases}
\end{displaymath}
This leads to the following results on
$\mathcal{H}_{\overline\partial_H}$:
\begin{itemize}
\item On $L^k \mathcal{E}^{n-q}$ for $0\leq k < q \leq n$,
  $\Delta_Q=\Delta_Q^J=\mathcal{L}_T=0$ by \eqref{eq:25},
\item On $\mathcal{E}^p$ for $p<n$, $\Delta_{\overline\partial_Q}=0$,
  so that by \eqref{eq:25} $\Delta_Q = -i T$ and
  $\mathbf{\Delta}_Q = \Delta_Q^2 = - T^2$,
\item On $L^q\mathcal{E}^{n-q}$ for $0< q \leq n$,
  $\Delta_{\partial_Q} =0$, so that $\Delta_Q = i T$ and
  $\mathbf{\Delta}_Q = \Delta_Q^2 = -T^2$,
\item On $\mathcal{E}^n$, $\delta_Q=0= \delta_Q^J=0$, so that
  $i_TD = T$ and $\mathbf{\Delta}_Q = - T^2$.
\end{itemize}
This yields Lemma~\ref{lemma:Delta_Q_on_H} because the isometry
$U = e^{i \frac{\pi}{4} (L+\Lambda)}$ mapping
$\mathcal{H}_{\overline\partial_H}$ to $\mathcal{H}=\ker D_H$ in
Proposition~\ref{prop:conjugation_D_H} commutes with
$\mathbf{\Delta}_Q$ and $T$ hence preserves the equation
$\mathbf{\Delta}_Q = -T^2$.

To conclude with this study of $\mathcal{H}$ we observe that it is
infinite dimensional since it contains images by $U$ of CR functions
(satisfying $ \overline\partial_H f=0$).

\smallskip We have obtained that
\begin{equation}
  \label{eq:28}
  \vartheta(t) = \tr(e^{tT^2} \mid \mathcal{H}^{ev}) - \tr(e^{tT^2 } \mid
  \mathcal{H}^{odd}) \,.
\end{equation}
Note that both traces converge since they are parts of traces of the
heat operators $e^{-t \mathbf{\Delta}_Q}$. The next step will be to
split this as a series of index using Fourier decomposition along the
circle action by $T$. We see this briefly and refer to
\cite[Section~4.1]{Rumin-Seshadri} for a more detailed discussion.

\smallskip

The flat bundle $V$ over $M$ associated to the unitary representation
$\rho : \pi_1(M) \rightarrow U(d)$ is the quotient of the trivial
bundle $\widetilde M \times \C^N$ by the deck transformations
$\gamma.(m,v) = (\tau(\gamma)m, \rho(\gamma) v)$. The circle action
$\varphi_t$ induced by $T$ on $M$ may be lifted to $V$ by parallel
transport using the flat connection $\nabla^\rho$. On $V$ we don't
have a circle action since
\begin{displaymath}
  \varphi_{2\pi} = \rho(f)^{-1}
\end{displaymath}
where $f= \varphi_{[0,2\pi]}(m)$ is the generic closed primitive
orbit. However we can split $V$ into irreducible representations $V^x$
on which
\begin{displaymath}
  \rho(f) = e^{2i \pi x}
\end{displaymath}
with $x \in ]0,1]$ by convention. We recover a circle action on $V^x$
by setting
\begin{displaymath}
  \psi_t = e^{i t x} \varphi_t \,.
\end{displaymath}
We can still perform a Fourier decomposition of sections
$s \in \mathbf{V}^x = \Gamma(M, V^x)$. One has
\begin{displaymath}
  s = \sum_{n\in \Z} \pi_n s \quad \mathrm{with} \quad \pi_n s =
  \frac{1}{2\pi} \int_0^{2\pi } e^{-int } \psi_t(s) dt \,.
\end{displaymath}
Since $\psi_t(\pi_ns) = e^{int} \pi_n s$
one gets that $\nabla^\rho_T(\pi_ns) = i(n-x) \pi_ns$ and the spectrum
of $iT= i\nabla^\rho_T $ on $\mathbf{V}^x$ is the shifted $x+\Z$. For
$\lambda = x+n$ we shall note
\begin{equation}
  \label{eq:29}
  \mathbf{V}_{\lambda} = \pi_{-n} (\mathbf{V}^x) = \mathbf{V}^x \cap \{ iT =
  \lambda\}. 
\end{equation}
Since $T$ commutes with all our geometric operators here as $D_H$, we
can split the $V$-valued bundle
\begin{displaymath}
  \mathcal{H}= \ker D_H = \bigoplus_{\lambda \in \spec(iT)} \mathcal{H}
  \cap \mathbf{V}_\lambda \,.
\end{displaymath}
Therefore \eqref{eq:28} eventually reads
\begin{align*}
  \vartheta(t)
  & = \sum_{\lambda \in \spec(iT)} \bigl(\dim(\mathcal{H}^{ev}
    \cap 
    \mathbf{V}_\lambda )- \dim(\mathcal{H}^{odd} \cap \mathbf{V}_\lambda)
    \bigr) e^{-t \lambda^2} \\
  & = \sum_{\lambda \in \spec (iT)} \ind (D_H^{ev} \mid \mathbf{V}_\lambda)
    e^{-t \lambda^2} = \vartheta^{top}(t) \,,
\end{align*}
with $D_H^{ev} : \Omega^{ev}H \rightarrow \Omega^{odd}H$. This is the
topological version of the $\vartheta$ series. We shall see that these
index can be interpreted as coming from operators and bundles over the
orbifold $N = M/\mathbf{S}^1$

\subsection{Index computations.}
\label{sec:index-computations}

\ \smallskip

The index in the previous formula can be computed using the index
theorem for $V$-(orbi)bundles over $V$-manifolds (orbifolds) as
developed by Kawasaki in \cite{Kawasaki}. One has first to interpret
the index of $D_H^{ev}$ on $\mathbf{V}_\lambda$ as the index of an
operator over the orbifold $N = M /\mathbf{S}^1$ acting on a
$V$-bundle.

\smallskip

We follow the discussion in \cite[section~5.2]{Rumin-Seshadri}. We
first introduce the relevant $V$-(orbi)bundle here. Given a point
$p \in N$, we consider the vector space $V_\lambda(p)$ of sections of
$V^x$ along the orbit $\mathbf{S}^1(p)$ in $M$ satisfying
$iT s = \lambda s$. Call $V_\lambda$ this family of spaces. One has by
definition
\begin{displaymath}
  \mathbf{V}_\lambda = \Gamma(N, V_\lambda) \,,
\end{displaymath}
and $V_\lambda$ is a genuine vector bundle of dimension $\dim V^x$
over the generic non singular points of $N$ since the circle action is
locally free.

We now describe the orbifold structure of $V_\lambda$ for
$\lambda = x+n$ near a singular point $p_j$ corresponding to an
exceptional closed primitive orbit $f_j$ of order $\alpha_j$. Locally
over $p_j$, the bundle $V^x$ is isomorphic to the quotient of the
trivial bundle $\C^n \times \R \times \C^{\dim V^x}$ by the deck
transformation
\begin{displaymath}
  F_j : (p,t,v) \mapsto (M_j(p), t + 2\pi / \alpha_j, \rho(f_j) v)
\end{displaymath}
where $M_j \in U(d)$ generates a cyclic group of order $\alpha_j$. Now
let $\widetilde V_\lambda$ be the trivial bundle over $\C^n$ whose
fiber over $p$ consists in functions
$s_p : \R \rightarrow \C^{\dim V_x}$ satisfying
$s_p(t) = e^{-i\lambda t} s_p(0)$. Since $(iT) s_p = \lambda s_p$, one
sees that a section $s : p \mapsto s_p$ of $\widetilde V_\lambda$ near
$0\in\C^n$ goes down to $\mathbf{V}_\lambda$ if it is invariant by the
deck transform $F_j$ above. This means that $(p, s_p(0))$ is invariant
by
\begin{equation}
  \label{eq:30}
  F_{j,\lambda} : (p, v) \mapsto (M_j(p) , e^{-2i \pi \lambda
    / \alpha_j} \rho(f_j)v) \,.
\end{equation}
Since $f_j^{\alpha_j} \sim f$ in $\pi_1(M)$, one has
$\rho(f_j)^{\alpha_j} = \rho(f) = e^{2i\pi x}$ and
$F_{j,\lambda}^{\alpha_j} = \id$.

This shows that $V_\lambda$ is a $V$-bundle (orbi-bundle) over $N$
since locally over $p_j$ it is the quotient of $\widetilde V_\lambda$
by the finite group $\Gamma_j \simeq \Z/\alpha_j \Z$ generated by
$F_{j,\lambda}$.

\smallskip

At this point we can identify
$\mathrm{ind}(D_H^{ev} \mid \mathbf{V}_\lambda)$ as being the index of
the Dirac operator $D_H^{ev}$ acting on sections of the $V$-bundle
$V_\lambda$ over $N$. Note that even if $d_H$ is not a complex, the
operator $D_H$ is elliptic on $N$ as
$D_H^2 = \Delta_H -L T + T\Lambda = \Delta_H$ up to order $0$ terms on
$V_\lambda$. More concretely from
Proposition~\ref{prop:conjugation_D_H}, $D_H$ is unitarily conjugated
to
$\slashed{\partial}_H = \overline\partial_H + \overline\partial_H^*$
hence
\begin{align*}
  \ind(D_H^{ev} \mid V_\lambda)
  & = \ind(\slashed{\partial}_H \mid \Omega^{ev}H \otimes
    V_\lambda) \\ 
  & = \sum_p (-1)^p \chi_{\overline\partial_H} (N,\Omega^{p,0}H \otimes
    V_\lambda) 
\end{align*}
where $\chi_{\overline\partial_H}(N,\Omega^{p,0}H \otimes V_\lambda)$
is the holomorphic Euler characteristic of the $V$-bundle
$\Omega^{p,0}H \otimes V_\lambda$. Recall that
$\overline\partial_H^2 =0$ on CR Seifert manifolds by
\eqref{eq:11}. From this discussion, one can compute these index using
Kawasaki's Riemann--Roch theorem for complex $V$-manifolds
\cite{Kawasaki}. It reads here
\begin{equation}
  \label{eq:31}
  \ind(D_H^{ev} \mid V_\lambda) = \langle \ch (V_\lambda) \wedge e (N) , [N]
  \rangle_{orb} \,. 
\end{equation}
This pairing on the orbifold $N$ splits into a usual smooth
contribution over the generic orbits of the characteristic classes and
an average of equivariant classes over the finite exceptional orbits;
see \cite{Kawasaki,Atiyah-Bott}.  Let
$N^* = N \setminus \cup_j \{p_j\}$ where $p_j= \pi (f_j)$ are the
singular points. One gets first
\begin{displaymath}
  \langle \ch (V_\lambda) \wedge e (N) , [N]
  \rangle_{orb} = \int_{N^*} \ch(V_\lambda) \wedge e(N) + \sum_j
  \frac{1}{\alpha_j} \sum_{k=1}^{\alpha_j-1} \tr (F_{j,\lambda}^k \mid
  \widetilde V_\lambda) \nu(f_j^k)  \,. 
\end{displaymath}
Here $\ch(V_\lambda)$ is the Chern character of $V_\lambda$ as seen
from $N$, $e(N)$ the Euler class of $N$ and $\nu(f_j^k)$ is the
Lefschetz index of the Poincaré return map along $f_j^k$. Since
$df_j^k \in U(n)$, one has $\nu(f_j^k)= 1$.  Also $e(N)$ being a top
order class, the index formula reduces to
\begin{equation}
  \label{eq:32}
  \langle \ch (V_\lambda) \wedge e (N) , [N]
  \rangle_{orb} = \dim(V^x) \chi(N^*) + \sum_j \dim\ker
  (F_{j,\lambda}- \id \mid \widetilde V_\lambda) \,,
\end{equation}
where $\chi(N^*)$ is the Euler characteristic of the punctured
manifold $N^*$ and
\begin{displaymath}
  \chi (N) = \int_{N^*} e(N) = \chi(N^*) + \sum_j \frac{1}{\alpha_j}
\end{displaymath}
is the rational Euler characteristic of the orbifold $N$.

\subsection{End of proof of the geometric formula for $\vartheta$.}
\label{sec:proof-geom-form}

\ \smallskip

We complete now the proof of the geometric identity
\begin{displaymath}
  \vartheta = \vartheta^{geo} = \chi(N^*) \chi^\theta_\rho(f) + \sum_j
  \chi^\theta_\rho(f_j) \,. 
\end{displaymath}
From \eqref{eq:31}, \eqref{eq:32} and the index series formula for
$\vartheta$ we know that
\begin{equation}
  \label{eq:33}
  \vartheta(t) = \sum_{\lambda \in \spec (iT)} \bigl( \dim(V^x) \chi(N^*) +
  \sum_j \dim\ker (F_{j,\lambda}- \id \mid \widetilde V_\lambda)
  \bigr) e^{-t\lambda^2} \,,
\end{equation}
where we recall that $\spec (iT)$ splits into $\Z + x$ on
$V^x = \ker (\rho(f) - e^{2i\pi x}\id)$.  Hence the smooth
contribution already reads
\begin{displaymath}
  \chi(N^*) \sum_{e^{2i \pi x} \in \spec (\rho(f))} \dim V^x \sum_{n\in \Z}
  e^{-t(n+x)^2} = \chi(N^*) \chi^\theta_\rho(f)(t) \,, 
\end{displaymath}
with $\chi^\theta_\rho$ as defined in \eqref{eq:6}.  In order to prove
\eqref{eq:7} it remains to evaluate the sums over the exceptional
orbits.

From \eqref{eq:30} one has
$F_{j,\lambda} = e^{-2i\pi \lambda/\alpha_j} \rho(f_j)$ on
$\widetilde V_\lambda$ at $p_j= \pi (f_j)$. Therefore
$F_{j,\lambda} = \id$ if $\rho(f_j) = e^{2i\pi
  \lambda/\alpha_j}$. Moreover since $\rho(f_j)^{\alpha_j} = \rho(f)$,
the spectrum of $\rho(f_j)$ on $V^x$ consists in complex numbers
$e^{2 i \pi x_{j,k}}$ with $\alpha_j x_{j,k} = x + n_{j,k}$ for some
integers $n_{j,k}$. Hence if $\lambda = x+ n$ one finds that
\begin{displaymath}
  e^{2i\pi x_{j,k}} = e^{2i\pi \lambda/\alpha_j}  \Leftrightarrow x_{j,k}
  \equiv \frac{x+n}{\alpha_j} \mathrm{mod}\ \Z \Leftrightarrow n \equiv
  n_{j,k}  \ \mathrm{mod}\
  \alpha_j \Z \,,
\end{displaymath}
so that $\lambda = \alpha_j x_{j,k} + \alpha_j p $ with $p \in
\Z$. Let
$V^{x_{j,k}}_{f_j} = \ker (\rho(f_j) - e^{2 i\pi x_{j,k}} \id)$. The
contribution of $f_j$ to \eqref{eq:33} reads
\begin{align*}
  \sum_{\lambda \in \spec (iT)} \dim
  & \ker (F_{j,\lambda} - \id \mid
    \widetilde V_\lambda) e^{-t \lambda^2}
  \\
  & = \sum_{e^{2i\pi x_{j,k}} \in \spec (\rho(f_j))} \dim V^{x_{j,k}}_{f_j}
    \sum_{p\in \Z}e^{-t \alpha_i^2 (x_{j,k} + p)^2} \\
  & = \sum_{e^{2i\pi x_{j,k}} \in \spec (\rho(f_j))} \dim V^{x_{j,k}}_{f_j}
    \theta(x_{j,k}, \alpha_j^2 t) = \chi_\rho^\theta (f_j)(t) \,.
\end{align*}
Note that we replace $\alpha_j $ by $2 \pi/ \ell (f_j)$ in the
definition of $\chi_\rho^\theta (f_j)(t)$ in order to preserve its
homogeneity in the rescaling $\theta \mapsto c^2 \theta$ and
$\ell (f_j) \mapsto c \ell(f_j)$.

\subsection{$\vartheta$ from the dynamical viewpoint.}
\label{sec:varth-from-dynam}

\ \smallskip

It remains to link the geometric formula for $\vartheta$ to its
dynamical one $\vartheta^{dyn}$. This is based on the usual Poisson
formula relating the Gaussian and Jacobi theta function.  It reads
\begin{equation}
  \label{eq:34}
  \theta(x,t) = \sum_{n\in \Z} e^{-t (n+x)^2}  = \sqrt { \frac{\pi}{t} }
  \sum_{n \in \Z} e^{2 i \pi n x} e^{-\pi^2 n^2 / t} \,.
\end{equation}
This leads to the following interpretation of
$\chi_\rho^\theta (\gamma)(t)$ for $\gamma= f$ and $f_i$. One gets
\begin{equation}
  \label{eq:35}
  \chi_\rho^\theta (\gamma)(t) = \frac{\ell(\gamma)}{\sqrt {4 \pi t}}
  \sum_{n\in \Z} \chi_\rho(\gamma^n) e^{-\ell^2(\gamma^n) / 4t}\,.
\end{equation}
Thus $\chi_\rho^\theta(\gamma)(t) $ is a pondered sum of holonomies of
$\rho$ along all the closed curves $\gamma^n$. More precisely, we
recall that the function
\begin{displaymath}
  k_t(y) = \frac{1}{\sqrt{4\pi t}} e^{-y^2/4t}
\end{displaymath}
is the heat kernel on $\R$. Now, given a random (Brownian) curve
$y(s)$ on the circle orbit $\gamma$, it is closed at time $t$ if its
lift $\widetilde y$ in $\R$ satisfies
$\widetilde y(t) - \widetilde y(0) = n \ell (\gamma)$. In such a case,
the rotation index of $y_{[0,t]}$ is $n$. Therefore its holonomy along
$\rho$ is $\chi_\rho(\gamma^n)$. Moreover the probability density of
such a displacement is
$k_t(\ell(\gamma^n)) = \frac{1}{\sqrt{4 \pi t}}
e^{-\ell^2(\gamma^n)/4t}$, that arises in \eqref{eq:35}.

\smallskip

Using \eqref{eq:35} in the geometric expression of $\vartheta$ gives
\begin{align*}
  \sqrt{4\pi t}\, \vartheta(t)
  =  \chi(N^*)
  & \ell(f) \sum_{n \in \Z} \chi_\rho(f^n) e^{- t
    \ell^2(f^n)/4t} \\
  & + \sum_i \ell(f_i) \sum_{n \in \Z} \chi_\rho (f_i^n)
    e^{-t \ell^2(f_i^n)/4t} \,.
\end{align*}
Now for $p \in \Z$, $f_i^{\alpha_i p } \sim f^p$ in $\pi_1(M)$ and
$\ell(f_i) = \ell(f) / \alpha_i$. Thus using
$\chi(N) = \chi(N^*) + \sum_i \frac{1}{\alpha_i}$ one finds that
\begin{align*}
  \sqrt{4\pi t}\, \vartheta(t)
  = & \chi(N) \ell(f) \sum_{n \in \Z} \chi_\rho(f^n) e^{- t
      \ell^2(f^n)/4t} \\
    & + \sum_i \ell(f_i) \sum_{n \notin \alpha_i \Z}
      \chi_\rho (f_i^n)  e^{-t \ell^2(f_i^n)/4t}  \,,
\end{align*}
which gives the required dynamical series \eqref{eq:8} over all
homotopy classes of closed orbits, their inverse, and the constant
loop.

\subsection{Zeta functions viewpoint.}
\label{sec:zeta-funct-viewp}

\ \smallskip

We now turn to identities on the contact analytic torsion from the
viewpoint of zeta functions. On the spectral side, let
\begin{displaymath}
  Z(s) = \sum_{k=0}^n (-1)^k (n+1-k) \zeta(\mathbf{\Delta}_Q\mid
  \mathcal{E}^k)(s)  \,,
\end{displaymath}
and on the dynamical side
\begin{displaymath}
  Z^{dyn} (s) = \sum_{\gamma \not=0 } \chi_\rho(\gamma) e(\gamma)
  |\ell(\gamma)|^{2s-1} \,,
\end{displaymath}
where $\gamma$ ranges over the homotopy classes of non trivial closed
orbits together with their inverse.

On the geometric side at last we consider
\begin{equation}
  \label{eq:36}
  Z^{geo}(s) = \chi(N^*) \chi_\rho^\zeta (f)(s)  + \sum_i \chi_\rho^\zeta
  (f_i)(s) 
\end{equation}
with
\begin{align*}
  \mathrm{order}(\gamma)^{2s} \chi_\rho^\zeta (\gamma) (s)
  = \sum_{e^{2i \pi x}\in \spec^*(\rho(\gamma))}
  & \dim V_\gamma^x  (\zeta(2s, x) + \zeta(2s,1-x)) \\
  & + 2 \zeta(2s) \dim V_\gamma^1   
\end{align*}
where $V^x_\gamma = \ker (\rho(\gamma) - e^{2i\pi x} \id) $, $\zeta$
is Riemann zeta function and
\begin{equation}
  \label{eq:37}
  \zeta(s,x) = \sum_{n \geq 0} \frac{1}{(n+x)^s}
\end{equation}
is Hurwitz zeta function.

\begin{thm}
  \label{thm:zeta-funct-viewp-1}
  On CR Seifert manifolds of any dimension the functions $Z$ and
  $Z^{geo}$ are meromorphic and equal, with a simple pole at $s=1/2$.

  The function $Z^{dyn}$ is meromorphic with a simple pole at $s=0$,
  and one has
  \begin{equation}
    \label{eq:38}
    \Gamma(s) Z(s) = \frac{2^{-2s}}{\sqrt \pi}
    \Gamma(\frac{1}{2} -s)  Z^{dyn} (s) \,.
  \end{equation}
  Moreover
  \begin{displaymath}
    \left\{
      \begin{gathered}
        Z(0) = \mathrm{Res}_{0} (Z^{dyn})= - \chi'(M,\rho) \\
        \mathrm{Res}_{1/2}(Z) = -\frac{1}{2\pi} Z^{dyn}(1/2)= \chi(N)
        \dim V \,.
      \end{gathered}
    \right.
  \end{displaymath}
  with
  \begin{displaymath}
    \chi'(M,\rho) = \sum_{k=0}^n (-1)^k (n+1-k)\dim H^k(M, \rho)\,,
  \end{displaymath}
  using the cohomology groups of the flat bundle $V$ over $M$.
\end{thm}

This extends results obtained in \cite{Rumin-Seshadri} in the three
dimensional case. Note also that Kitaoka computed in \cite{Kitaoka}
the spectral series $Z(s)$ on Lens spaces endowed with their symmetric
metric and found it reduces to a Hurwitz zeta function.

\subsection{Proof of Theorem~\ref{thm:zeta-funct-viewp-1}}
\label{sec:proof-theor-refthm:z}

\ \smallskip

We know from the spectral expression of $\vartheta$ that
\begin{displaymath}
  \vartheta(t) = \sum_{k=0}^n (-1)^k (n+1-k) \dim \ker (\mathbf{\Delta}_Q \mid
  \mathcal{E}^k) + 0 (e^{-Ct})  
\end{displaymath}
for some $C >0$ when $t \rightarrow +\infty$. According to
Proposition~\ref{prop:contact-complex-1} the limit constant is the
topological number $\chi'(M,\rho)$.

From $\vartheta = \vartheta^{geo}$, one gets using
$t\rightarrow +\infty$ that
\begin{equation}
  \label{eq:39}
  \chi'(M, \rho) = \chi(N^*) \dim \ker (\rho(f) - \id) + \sum_{i} \dim \ker
  (\rho(f_i)  - \id)\,.
\end{equation}
Whereas when $t \rightarrow 0^+$, the dynamical expression yields
\begin{displaymath}
  \vartheta(t) = \sqrt{\frac{\pi}{t}} \chi(N) \dim V + O(e^{-C/t}) \,.
\end{displaymath}

Let consider the function
\begin{displaymath}
  \vartheta_0 (t) = \vartheta(t) - c_0 \mathbf{1}_{[1,+\infty[}(t) -
  \frac{c_{-1/2}}{\sqrt t} \mathbf{1}_{]0,1]}(t) 
\end{displaymath}
with
\begin{displaymath}
  c_0 = \chi'(M,\rho) \quad \mathrm{and} \quad c_{-1/2} = \sqrt \pi \chi(N)
  \dim V \,.
\end{displaymath}
From the previous discussion, the integral
\begin{displaymath}
  I(s) = \int_0^{+\infty} \vartheta_0(t) t^{s-1} dt 
\end{displaymath}
coming from Mellin transform defines a holomorphic function over
$\C$. Writing
\begin{displaymath}
  \vartheta_0(t) = \vartheta(t) - c_0 + (c_0 - \frac{c_{-1/2}}{\sqrt t})
  \mathbf{1}_{]0,1]}(t) 
\end{displaymath}
one finds first for $\Re(s) > 1/2$ that
\begin{displaymath}
  I(s) = \Gamma(s) Z(s) + \frac{c_0}{s}- \frac{c_{-1/2}}{s-1/2} =  \Gamma(s)
  Z^{geo}(s) + \frac{c_0}{s}- \frac{c_{-1/2}}{s-1/2} \,.
\end{displaymath}
Using
\begin{displaymath}
  \vartheta_0(t) = \vartheta^{dyn}(t) - \frac{c_{-1/2}}{\sqrt t} +
  (\frac{c_{-1/2}}{\sqrt t}
  - c_0) \mathbf{1}_{]1,+\infty]}(t) 
\end{displaymath}
yields for $\Re(s) < 0$ that
\begin{displaymath}
  I(s) = \frac{2^{-2s}}{\sqrt \pi} \Gamma(\frac{1}{2} -s) Z^{dyn}(s) +
  \frac{c_0}{s} - \frac{c_{-1/2}}{s-1/2} \,.
\end{displaymath}
This proves Theorem~\ref{thm:zeta-funct-viewp-1}.

\subsection{The contact analytic torsion from the Lefschetz and
  dynamical viewpoints.}
\label{sec:lefsch-type-form}

\ \smallskip

Recall that from \eqref{eq:3} the contact analytic torsion is defined
by $T_Q(M,\rho) = \exp(-Z'(0)/2)$. We shall now compute it using the
previous formulae. This extends results obtained in
\cite{Rumin-Seshadri} in the three dimensional case.

\begin{cor}
  \label{cor:lefsch-type-form-2}
  Let $M$ be a CR Seifert manifold of any dimension endowed with a
  unitary representation $\rho : \pi_1 (M) \rightarrow U (N)$. Let
  \begin{displaymath}
    V^1_f = \ker (\rho(f) - \id) \quad \mathrm{and} \quad V^1_{f_i} = \ker
    (\rho(f_i) - \id) \,,
  \end{displaymath}
  and denote by $\rho(f)^\bot$ and $\rho(f_i)^\bot$ the restriction of
  these holonomies to respectively $(V^1_f)^\bot$ and
  $(V^1_{f_i})^\bot$.

  Then it holds that
  \begin{displaymath}
    T_Q(M,\rho) = (2\pi)^{\chi'(M,\rho)} |\det(\rho(f)^\bot -
    \id)|^{\chi(N^*)} \prod_i\frac{| \det (\rho(f_i)^\bot -
      \id)|}{\alpha_i^{\dim(V^1_{f_i}) }}\,.
  \end{displaymath}
  Moreover one has
  \begin{displaymath}
    -2 \ln (T_Q(M, \rho)) = Z'(0) = \lim_{s \rightarrow 0} \big(Z^{dyn}(s) +
    \frac{\chi'(M,\rho)}{s} \bigr)\,.
  \end{displaymath}
\end{cor}

The first expression coincides with that found, via topological
methods, for the Reidemeister--Franz torsion by Fried
\cite[p.~198]{Fried}, in the case of an \emph{acyclic representation},
i.e. $H^*(M,\rho) = 0$. Therefore, the contact analytic torsion also
coincides with the (Riemannian) Ray--Singer analytic torsion in that
case, from works of Cheeger and M\"uller \cite{Cheeger,Muller}. The
only new factor for general representations is the cohomological term
$(2 \pi)^{\chi'(M,\rho)}$.

The dynamical expression for the analytic torsion also extends a
result proved by Fried in the acyclic case on Seifert manifolds
\cite{Fried87, Fried}.
We have there
\begin{displaymath}
  Z'(0)=  Z^{dyn}(0) = \sum_{\gamma\not= 0} \chi_\rho(\gamma)
  e(\gamma) /|\ell(\gamma)| \,, 
\end{displaymath}
which is known as the total twisted Fuller measure of periodic
orbits. It has a formal invariance by deformation of the flow, as long
as the orbit periods stay bounded; see \cite[Section~4]{Fried87}.

\begin{proof}[Proof of Corollary~\ref{cor:lefsch-type-form-2}]
  \
  
  Since $Z =Z^{geo}$ is an explicit function by
  Theorem~\ref{thm:zeta-funct-viewp-1}, we compute $(Z^{geo})'(0)$. We
  follow the proof in \cite[p.~771]{Rumin-Seshadri} that we recall for
  completeness.

  From \cite[p.~271]{WW}, one has for $x \in ]0,1[$,
  $\zeta(0, x) = \frac{1}{2} -x $, hence
  $\zeta(0,x) + \zeta(0,1-x) = 0$. Also by Lerch's formula
  $\partial_s \zeta(s,x)_{s=0} = \ln \Gamma(x) - \frac{1}{2} \ln
  (2\pi)$, it holds that $\zeta'(0) = - \frac{1}{2} \ln (2\pi)$ and
  \begin{align*}
    \partial_s \zeta(s,x) + \partial_s \zeta(s,1-x)
    &= \ln (\Gamma(x)\Gamma(1-x) /2\pi) = - \ln (2 \sin(\pi x)) \\
    & = -\ln |1 - e^{2i\pi x}| \,.
  \end{align*}
  Hence from \eqref{eq:36}, one finds that
  \begin{align*}
    \chi_\rho^\zeta (f)'(0)
    & = - 2 \sum_{e^{2i\pi x }\in \spec \rho(\gamma)^\bot}
      \dim V_f^x \ln | 1 - e^{2 i \pi x} | + \dim V^1_f \ln 2\pi \\
    & = - 2 \ln |\det ( \rho(f)^\bot - \id)| -2 \dim V^1_f \ln (2 \pi) \,,
  \end{align*}
  and for the exceptional fibers,
  \begin{align*}
    \chi_\rho^\zeta (f_i)'(0) + 2 \ln(\alpha_i) \chi_\rho^\zeta (f_i)(0) 
    & = - 2 \ln |\det ( \rho(f_i)^\bot - \id)| -2 \dim V^1_{f_i} \ln (2
      \pi) \,. 
  \end{align*}
  Summing these results using the formula \eqref{eq:36} for $Z^{geo}$
  and \eqref{eq:39} gives the first result for $T_Q(M,\rho)$.

  For the second one, it is useful to observe that
  \begin{equation}
    \label{eq:40}
    2 \Gamma(2s) \cos (\pi s) Z = Z^{dyn} \,,
  \end{equation}
  as comes by multiplying \eqref{eq:38} by $\Gamma(s+1/2)$ and using
  the classical identities
  \begin{displaymath}
    \Gamma(s) \Gamma(s + \frac{1}{2}) = 2^{1-2s} \sqrt \pi \Gamma(2s)\quad
    \mathrm{and} \quad \Gamma(s+ \frac{1}{2}) \Gamma(s - \frac{1}{2}) =
    \frac{\pi}{ \cos \pi s}\,.
  \end{displaymath}
  One knows from Theorem~\ref{thm:zeta-funct-viewp-1} that
  $Z(0) = - \chi'(M,\rho)$. Hence developing \eqref{eq:40} when
  $s\rightarrow 0$ gives the dynamical formula for $T_Q(M,\rho)$.
\end{proof}

\section{Results on the contact eta trace.}
\label{sec:results-cont-sign}

We now turn to the study of the contact eta invariant and trace as
introduced in Section~\ref{sec:around-contact-eta}. Following
\eqref{eq:9} we consider the contact signature operator on contact
manifolds of dimension $2n+1= 4k-1$
\begin{displaymath}
  S_Q =
  \begin{cases}
    *D + (d_Q + \delta_Q) \sigma \delta_Q & \mathrm{on}\
    \mathcal{E}^{2k-1}\\
    (d_Q+\delta_Q)\sigma (d_Q+\delta_Q) & \mathrm{on} \
    \bigoplus_{1\leq p\leq k-1 } \mathcal{E}^{2k-1-2p}
  \end{cases}
\end{displaymath}
where $\sigma= (-1)^p J$ on $\mathcal{E}^{2p}$ and
$\mathcal{E}^{2p-1}$. As claimed the advantage of this choice over
others lies in the following algebraic properties.
\begin{prop}
  \label{prop:S_Q}
  \

  $\bullet$ The operator $S_Q$ is symmetric on any contact manifold of
  dimension $4k-1$ endowed with a calibrated complex structure $J$.

  $\bullet$ If moreover $J$ is integrable and invariant by $T$, then
  it holds that $S_Q^2 = \mathbf{\Delta}_Q$ and
  \begin{displaymath}
    S_Q = \sigma T + P \ \mathrm{with}\ P\sigma = -\sigma P\
    \mathrm{and}\ [T,\sigma]= [P,T]= 0\,.
  \end{displaymath}

  $\bullet$ The non-zero spectrum of $S_Q$ splits as follows
  \begin{displaymath}
    \spec^*(S_Q) = \spec^*(*D) \bigsqcup_{0\leq p\leq
      k-1} \spec^* \bigl(\sigma \Delta_Q\mid \mathcal{E}^{2p} \bigr)\,. 
  \end{displaymath}
\end{prop}

\begin{proof}
  $\bullet$ The first statement follows easily from the definition and
  the facts that by \eqref{eq:18}
  \begin{displaymath}
    (-1)^k*D = J(i_TD) = JT - J d_Q J^{-1}\delta_Q J  
  \end{displaymath}
  with $J^* = (-1)^pJ$ on $\mathcal{E}^p$ and $T^* = -T^J$ on
  calibrated $J$ by \eqref{eq:21}. Note that using the $\sigma$
  symmetry this writes
  \begin{equation}
    \label{eq:41}
    *D = \sigma T - \sigma^{-1}d_Q \sigma \delta_Q \sigma \,.
  \end{equation}

  $\bullet$ By Proposition~\ref{prop:Laplacian-zero-torsion} one knows
  that $\Delta_Q$ commutes with $\sigma$ on $\mathcal{E}^p$ for
  $p< n= 2k-1$ on CR manifolds with transverse symmetry. Then using
  that the sequence $(d_Q,D)$ is a complex, one finds on
  $\mathcal{E}^{2k-1-2p}$ for $p\geq 2$ that
  \begin{align*}
    S_Q^2
    & = (d_Q + \delta_Q) \sigma (d_Q + \delta_Q)^ 2 \sigma (d_Q +
      \delta_Q) \\
    & = (d_Q + \delta_Q) \sigma \Delta_Q \sigma (d_Q + \delta_Q) =
      (d_Q + \delta_Q) \Delta_Q (d_Q + \delta_Q ) \\
    & = \Delta_Q^2 = \mathbf{\Delta}_Q \,.
  \end{align*}
  In degree $2k-3$
  \begin{align*}
    S_Q^2
    & = (*D
      + (d_Q+ \delta_Q) \sigma \delta_Q) (d_Q \sigma d_Q)  \\
    & +  (d_Q + \delta_Q) \sigma (d_Q+ \delta_Q)( (d_Q+
      \delta_Q) \sigma \delta_Q +  \delta_Q \sigma d_Q) \\
    & = (d_Q+\delta_Q)\sigma \Delta_Q \sigma d_Q +(d_Q+ \delta_Q) \sigma
      \Delta_Q \sigma \delta_Q \\
    & = (d_Q+\delta_Q) \Delta_Q d_Q + (d_Q + \delta_Q) \Delta_Q \delta_Q=
      \Delta_Q^2 = \mathbf{\Delta}_Q
    \\
  \end{align*}
  In degree $2k-1$
  \begin{align*}
    S_Q^2
    & = (*D)^2 + ((d_Q + \delta_Q) \sigma \delta _Q)( d_Q \sigma \delta_Q) +
      (d_Q+ \delta_Q) \sigma (d_Q + 
      \delta_Q )\delta_Q\sigma \delta_Q\\
    & = D^* D + (d_Q + \delta_Q) \sigma \Delta_Q \sigma \delta_Q \\
    & = D^* D + (d_Q + \delta_Q) \Delta_Q \delta_Q = D^* D +
      (d_Q\delta_Q)^2 = \mathbf{\Delta}_Q \,.
  \end{align*}

  We determine the invariant part $(S_Q)^\sigma$ of $S_Q$ through
  $\sigma$. One observes first from \eqref{eq:23} that
  $d_Q \sigma d_Q$ adds $(1,1)$ to the bidegree, hence preserves $J$
  and anti-commutes with $\sigma$. The same holds for
  $\delta_Q \sigma \delta_Q$. Therefore
  \begin{equation}
    \label{eq:42}
    (S_Q)^\sigma =
    \begin{cases}
      (*D + d_Q \sigma \delta_Q)^\sigma & \mathrm{on} \ \mathcal{E}^{2k-1}\\
      (d_Q \sigma \delta_Q + \delta_Q \sigma d_Q)^\sigma & \mathrm{on}
      \ \bigoplus_{1\leq p \leq k-1} \mathcal{E}^{2k-1 - 2p} \,.
    \end{cases}
  \end{equation}
  Then \eqref{eq:41} gives in degree $2k-1$ that
  \begin{displaymath}
    (S_Q)^\sigma = (\sigma T - \sigma^{-1}d_Q \sigma \delta_Q \sigma + d_Q
    \sigma \delta_Q)^\sigma = \sigma T \,. 
  \end{displaymath}

  In degree $2k-1 - 2p$ with $p\geq 1$, we know from \eqref{eq:20}
  that $ T = \delta_Q^J d_Q + d_Q \delta_Q^J$, thus
  \begin{displaymath}
    \sigma T = \delta_Q \sigma d_Q + \sigma^{-1} d_Q \sigma \delta_Q \sigma \,.
  \end{displaymath}
  By \eqref{eq:42}, we look for
  \begin{displaymath}
    (d_Q \sigma \delta_Q + \delta_Q \sigma d_Q)^\sigma
    = (\sigma T - \sigma^{-1} d_Q \sigma \delta_Q \sigma + d_Q
    \sigma \delta_Q)^\sigma = \sigma T \,,
  \end{displaymath}
  as needed.
  
  $\bullet$ We come back to the proof of Proposition. To relate the
  spectrum of $S_Q$ to the one of $*D$ we first observe that $S_Q= *D$
  on $\mathcal{E}^n_Q= \mathcal{E}^n \cap \ker \delta_Q$. Then
  \begin{displaymath}
    \spec^*(S_Q) = \spec^*(*D) \bigsqcup \spec^* (S_Q \mid
    (\mathcal{E}^n_Q)^\bot )\,.
  \end{displaymath}
  Moreover on $ (\mathcal{E}^n_Q)^\bot$,
  $(d_Q+\delta_Q) S_Q = \Delta_Q \sigma (d_Q + \delta_Q)$, with the
  convention that $d_Q = 0$ in degree $n$.  Therefore
  \begin{displaymath}
    \spec^* (S_Q \mid (\mathcal{E}^n_Q)^\bot ) = \bigsqcup_{0\leq p\leq
      k-1} \spec^* \bigl(\sigma \Delta_Q\mid \mathcal{E}^{2p} \bigr)\,. 
  \end{displaymath}
\end{proof}

From this, we can compare the eta trace functions of $S_Q$ and
$*D$. It holds on CR Seifert manifolds that
\begin{equation}
  \label{eq:43}
  \eta(S_Q)(s) = \eta(*D)(s) + \sum_{0\leq p\leq q} \eta(\sigma \Delta_Q
  \mid \mathcal{E}^{2p}) (s)\,.
\end{equation}
Moreover $\sigma \Delta_Q$ splits through bidegree and
\begin{displaymath}
  \eta(\sigma \Delta_Q \mid \mathcal{E}^{2p})(s) = \sum_{a+b = 2p}
  (-1)^p i^{a-b}
  \zeta(\Delta_Q \mid \mathcal{E}^{a,b})(s) \,.
\end{displaymath}
For positive hypoelliptic operators like $\Delta_Q$, one knows that
\begin{displaymath}
  \zeta(\Delta_Q)(0) + \dim \ker \Delta_Q 
\end{displaymath}
is the constant term in the development of $e^{-t \Delta_Q}$ and is
given by the integral over $M$ of some universal polynomial of local
invariant of the pseudo-hermitian metric; see \cite[Section
3.1]{Rumin-Seshadri} or \cite{Albin-Quan} for discussion and
references. Moreover the kernels of $\Delta_Q$ are isomorphic to the
cohomology groups $H^*(M,\rho)$. Hence we finally get the following
\begin{prop}
  \label{prop:eta-etaS_Q}
  On a CR Seifert manifold of given dimension $4k-1$, it holds that
  \begin{displaymath}
    \eta(S_Q)(0) - \eta(*D)(0) + \sum_{0 \leq i+j=  2p \leq 2 (k-1)}
    (-1)^p i^{a-b} \dim H^{a,b}(M,\rho)  
  \end{displaymath}
  is the integral over $M$ of some universal polynomial of local
  invariants of the metric.
\end{prop}

\begin{rems}
  \label{rems:SQ_eta}
  In dimension three with a trivial representation $\rho$,
  $H^0(M,\rho) = \R$ is given by the constant functions. Hence the
  cohomological sum is $1$ and $\eta(S_Q)(0) - \eta(*D)(0)$ is
  \emph{not} given by a local term a priori. The two eta invariants
  are not equivalent up to local terms.

  Note also that in dimension $3$, one has $S_Q = *D + d_Q \delta_Q$,
  as studied for instance in previous work \cite{BHR}, but the
  expression differs in higher dimensions.
\end{rems}

The advantage of working with $S_Q$ instead of $*D$ or
$*D \pm d_Q\delta_Q$ in general lies in its spectral symmetry with
respect to $\sigma$.

\subsection{The torsion contact trace from the topological viewpoint.}
\label{sec:tors-cont-trace}

\ \smallskip

We start the study of the eta trace spectral series involved in the
eta invariant $S_Q$. From \eqref{eq:10} it reads
\begin{displaymath}
  \vartheta_S (t) = \tr (\sqrt t S_Q e^{-t\mathbf{\Delta}_Q})\,.
\end{displaymath}

\smallskip

We complete arguments already sketched in
Section~\ref{sec:contact-eta-trace}. From Proposition~\ref{prop:S_Q},
one has $\sigma P S_Q = -S_Q \sigma P$ with $\sigma P = -P
\sigma$. Therefore the spectrum of $S_Q = P + \sigma T$ is symmetric
with respect to zero, except on $\mathcal{H}_S = \ker P$ and the eta
trace $\vartheta_S$ retracts on it with $S_Q = \sigma T$. Hence
\begin{displaymath}
  \mathbf{\Delta}_Q = S_Q^2= -T^2 \quad
  \mathrm{on} \quad \mathcal{H}_S \,. 
\end{displaymath}
One can split the domain $\mathcal{E}_S$ of $S_Q$ and $\mathcal{H}_S$
with respect to the action of the involution $\tau= i\sigma = \pm 1
$. We have
\begin{displaymath}
  \vartheta_S(t)  = -\sqrt t(\tr (iT e^{tT^2} \mid \mathcal{H}_S^+) - \tr(iT
  e^{tT^2 } \mid \mathcal{H}_S^-))\,.
\end{displaymath}
On CR Seifert manifolds, the $V$-valued forms in $\mathcal{H}_S$ split
under Fourier decomposition as in
Sections~\ref{sec:from-spectr-topol}--\ref{sec:index-computations}. Hence
we obtain the identity.
\begin{thm}
  \label{thm:eta=eta_top}
  One has on CR Seifert manifolds
  \begin{displaymath}
    \vartheta_S(t) = \vartheta_S^{top}(t) =- \sqrt t \sum_{\lambda
      \spec(iT)} \ind( P^+ \mid V_\lambda) \lambda e^{-t\lambda^2} \,,
  \end{displaymath}
  where $P^+ = P : \mathcal{E}_S^+ \rightarrow \mathcal{E}_S^-$.
\end{thm}

We shall now compute the index involved.

\subsection{From the topological to the geometric expression for
  $\vartheta_S$.}
\label{sec:from-topol-geom}

\ \smallskip

Let $*_H$ denotes the $*$ operator on the Hermitian space $H$. From
\cite[Thm. 2]{Weil} one has
\begin{displaymath}
  *_H L^r \alpha = L^r((-1)^{p(p+1)/2} J \alpha) = L^r(\sigma \alpha)
\end{displaymath}
on $\Omega^p_0 H = \mathcal{E}^p$ when $p+2r = n$.  That means that
the involution $\tau = i\sigma$ on $\mathcal{E}_S$ is conjugated to
$\tau_H = i*_H$ on $\Omega^nH$ through the Lefschetz decomposition
\begin{displaymath}
  \mathcal{E}_S \simeq \Omega^n H\quad \mathrm{with} \quad
  \mathcal{E}_S^{\tau = \pm 1} \simeq (\Omega^nH)^{\tau_H = \pm 1} \,.
\end{displaymath}
From these isomorphisms
$P^+ : \Omega^{n,+} H \rightarrow \Omega^{n,-}H$.  Hence its elliptic
symbol class as seen from $N$ is associated to the signature index on
$N$; see \cite[Section~6]{At-Pa-SiIII}. Let then
$\mathcal{L}(N)_{orb}$ denotes the Hirzebruch $L$-genus of the
orbifold $N$. As in Section~\ref{sec:index-computations} we have by
Kawasaki index theorem and \cite{At-Pa-SiIII}
\begin{displaymath}
  \ind(P^+ \mid V_\lambda) =  \langle \ch(V_\lambda) \wedge
  \mathcal{L}(N)_{orb}, [N] \rangle_{orb} \,,
\end{displaymath}
and $\vartheta_S(t)$ can be written
\begin{equation}
  \label{eq:44}
  \vartheta_S(t) = - \langle  \ch(V)^{\theta}_{odd}(t)\wedge
  \mathcal{L}(N)_{orb}, [N] \rangle_{orb} 
\end{equation}
using the notation
\begin{equation}
  \label{eq:45}
  \ch(V)^{\theta}_{odd}(t) = \sum_{\lambda \in \spec(iT)} \sqrt t \lambda
  \ch(V_\lambda)e^{-t \lambda^2} \,.
\end{equation}
We need to study this $\theta$-regularised Chern character of $V$ as
seen from $N$. It is also useful to consider its even version
\begin{equation}
  \label{eq:46}
  \ch(V)^{\theta}_{ev}(t) = \sum_{\lambda \in \spec(iT)}
  \ch(V_\lambda)e^{-t \lambda^2} \,.
\end{equation}
Note that its zero degree part already showed up in \eqref{eq:31} in
the study of the topological series for the torsion function, that
writes
\begin{displaymath}
  \vartheta^{top} (t) = \langle \ch^\theta_{ev}(V)(t) \wedge e(N)_{orb} , [N]
  \rangle_{orb} \,.
\end{displaymath}
From \eqref{eq:44} we need now to compute these differential forms in
their whole.

\smallskip

To compute the Chern character of $V_\lambda$, we use the following
twisted connection
\begin{displaymath}
  \nabla^\lambda = \nabla^\rho + i \lambda \theta \,,
\end{displaymath}
where $\nabla^\rho$ is the flat connection on $V$. Since
$\nabla^\lambda_T s = 0$ on sections of $\mathbf{V}_\lambda$, this
connection goes down on $M$ as a connection on $V_\lambda$. Its
curvature form is given for $X,Y \in H$ by
\begin{align*}
  R_{\nabla^\lambda}(X,Y)
  & = \nabla^\lambda_X \nabla^\lambda_Y - \nabla^\lambda_Y \nabla^\lambda_X
    - \nabla^\lambda_{[X,Y]} \\ 
  & = R_{\nabla^\rho} (X,Y) + i \lambda d\theta(X,Y) \\
  & = i \lambda d\theta(X,Y)\,.
\end{align*}
Hence the bundle $(V_\lambda, \nabla_\lambda)$ has curvature form
$\Omega_\lambda= i \lambda d\theta \otimes \id_{V_\lambda}$. Let then
\begin{displaymath}
  \mathbf{c} = -\frac{d \theta}{2\pi} \,.
\end{displaymath}
Note that $\c = c_1(L)$ is the first Chern class of the line bundle
$L $ over $N$ whose circle bundle is $M$. Using the $V$-bundle
structure of $V_\lambda$ as given in \eqref{eq:30}, one obtains that
\begin{displaymath}
  \ch (V_\lambda) = \tr(e^{\frac{i \Omega_\lambda}{2\pi}}) = \dim V^x
  e^{\lambda \c}  
\end{displaymath}
over smooth points of $N$, while the equivariant Chern character at
$f_j^r$ is
\begin{displaymath}
  \ch(V_\lambda)(f_j^r) = e^{- 2i \pi r\lambda/\alpha_j }
  \chi_\rho(f_j^r \mid V^x)\,.
\end{displaymath}

This leads to the geometrical and dynamical expressions for the
$\theta$-regularised Chern characters $\ch(V)^\theta_{ev}$ and
$\ch(V)^\theta_{odd}$ using the classical Jacobi theta function
$\theta(x,t) = \sum_{n \in \Z} e^{-t(n+x)^2}$ and Poisson formula
\eqref{eq:34}.
\begin{prop}
  \label{prop:ch_theta}
  $\bullet$ Over smooth orbits it holds that
  \begin{align*}
    \ch(V)^\theta_{ev} (t)
    & = \sum_{\lambda \in \spec(iT)} \dim V^x e^{\lambda \c - t
      \lambda^2} \\
    & = \sum_{e^{2i\pi x} \in \spec (\rho(f))}\dim V^x
      e^{\c^2/4t}\theta (x - \frac{\c}{2t},t) \\
    & = \frac{\ell(f)}{\sqrt{4\pi t}} \sum_{n \in \Z} \chi_\rho(f^n) e^{-
      (\ell(f^n) + i \c)^2/ 4t}
  \end{align*}
  and
  \begin{align*}
    \ch(V)^\theta_{odd}(t)
    & = \sqrt t \sum_{\lambda \in \spec(iT)} \dim
      V^x \lambda e^{\lambda \c -t \lambda^2} \\
    & = \sqrt t \frac{d}{d\c} \ch(V)^\theta_{ev} (t)\\
    & = - \frac{i \ell(f)}{4t\sqrt{\pi}} \sum_{n \in \Z}
      \chi_\rho(f^n)(\ell(f^n)  + i \c) e^{- (\ell(f^n) + i \c)^2/ 4t} , 
  \end{align*}
  in terms of formal derivation as a polynomial in $\c$.

  $\bullet$ At a singular orbit $f_j^r$, the equivariant
  $\theta$-Chern characters writes
  \begin{align*}
    \ch(V)^\theta_{ev}(f_j^r)(t)
    & = \sum_{\lambda \in \spec(iT)} \chi_\rho(f_j^r \mid V^x) e^{-2i\pi r
      \lambda/\alpha_j -t \lambda^2}\\
    & = \sum_{e^{2i\pi x }\in \spec(\rho(f))} \chi_\rho(f_j^r \mid V^x)
      e^{-\pi^2 r^2/t\alpha_j^2} \theta(x  + i
      \frac{\pi r}{t \alpha_j},t) \\ 
    & = \frac{\alpha_j \ell(f_j)}{\sqrt{4 \pi t}}\sum_{n \in \Z}
      \chi_\rho(f_j^{r + n \alpha_j}) e^{-\ell(f_j^{r+ n\alpha_i})^2/4t }
  \end{align*}
  and
  \begin{align*}
    \ch(V)^\theta_{odd}(f_j^r)(t)
    & = \sqrt t \sum_{\lambda \in \spec(iT)} \chi_\rho(f_j^r \mid V^x)
      \lambda e^{-2i\pi r \lambda/\alpha_j -t \lambda^2}\\
    & = - \frac{i \alpha_j \ell(f_j)}{4t \sqrt{\pi}}\sum_{n \in \Z}
      \chi_\rho(f_j^{r+ n \alpha_j})\ell(f_j^{r + n \alpha_j}) 
      e^{- \ell(f_j^{r + n\alpha_j})^2/ 4t}  \,.
  \end{align*}
\end{prop}
Note the similar expressions for regular and discrete orbits. All the
homotopy classes of closed orbits of the circle action, together with
their opposite, enter only once in the various contributions. We also
observe that the spectral--dynamical duality in our Selberg--type
trace formulae actually shows up at the level of these
$\theta$-regularised Chern characters, and only depends on the
classical Poisson formula for the heat kernel on the circle.

\smallskip

We can use this into the index series \eqref{eq:44}
\begin{displaymath}
  \vartheta_S(t) = - \langle  \ch(V)^\theta_{odd}(t)\wedge
  \mathcal{L}(N)_{orb}, [N] \rangle_{orb} \,.
\end{displaymath}
We recall that from Kawasaki and Atiyah--Bott works \cite{Kawasaki}
and \cite[Section~6]{Atiyah-Bott}, this orbifold pairing decomposes
into a classical smooth integral of characteristic classes and a
discrete singular sum
\begin{align}
  \label{eq:47}
  \langle  \ch(V)^{\theta}_{odd}(t)\wedge \mathcal{L}(N)_{orb}, [N]
  \rangle_{orb} = \langle 
  & \ch(V)^\theta_{odd} \wedge \mathcal{L}(N) , [N_{smooth}] \rangle \\
  & + \sum_i \frac{1}{\alpha_i} \sum_{r=1}^{\alpha_i-1}
    \ch(V)^\theta_{odd}(f_i^r) \nu(f_i^r)  \,,\nonumber
\end{align}
where, for $\gamma= f_j^n$ with $n\not\equiv 0 \mod \alpha_j$,
\begin{displaymath}
  \nu(\gamma)= i (-1)^k \prod_{m=1}^{2k-1} \cot (\theta_m(\gamma)/2)
\end{displaymath}
for angles $\theta_m(\gamma)$ associated to the unitary return map
$d \tau_\gamma$ along $\gamma $ in $H_{p_j}$ and $\dim N = 4k-2$. Note
that $e^{2i \pi \theta_m(f_j)}$ are primitive $\alpha_i$-th roots of
unity since $f_i^{\alpha_i} = f$ generates a locally free circle
action.

\smallskip

Using Proposition~\ref{prop:ch_theta}, we obtain the dynamical
expression for the contact eta function.
\begin{thm}
  \label{thm:eta_dyn}
  On CR Seifert manifolds, one has
  \begin{displaymath}
    \vartheta_S(t) = \vartheta_S^{dyn}(t)= \frac{1}{\sqrt {4 \pi}} \sum_\gamma
    \chi_\rho(\gamma) \sigma(\gamma)(t)
  \end{displaymath}
  where powers of the generic orbit $\gamma= f^n$, with $n \in \Z$,
  contribute to
  \begin{displaymath}
    \sigma(\gamma)(t) =
    \frac{i \ell(f)}{2t} \, \bigl\langle (\ell(\gamma) + i \c) e^{-(\ell
      (\gamma) + i \c)^2 /4t} \wedge \mathcal{L}(N), [N_{smooth}]
    \bigr\rangle \,, 
  \end{displaymath}
  while powers of the singular orbits $\gamma= f_j^p$, with
  $p \not\equiv 0 \mod\alpha_j$, contribute to
  \begin{displaymath}
    \sigma(\gamma)(t) = \frac{i \ell(f_j)}{2t} \ell(\gamma)
    e^{-\ell(\gamma)^2/4t} \, \nu(\gamma) \,.
  \end{displaymath}
\end{thm}

\begin{rem}
  \label{rem:sigma_symmetry}
  We observe that $\sigma(\gamma)(t) $ is real. Indeed for the smooth
  component, only the odd part in $\c$ contributes to it since
  $\mathcal{L}(N)$ has components in degrees $4p$ while
  $\dim N = 4k -2$. Then $\sigma(\gamma^{-1})(t) = \sigma(\gamma)(t)$
  so that $\vartheta_S(t) $ is real on unitary representations.
\end{rem}

\subsection{Applications to the contact eta function $\eta(S_Q)(s)$.}
\label{sec:appl-cont-eta-1}

\ \smallskip

In order to get applications to the contact eta invariant we need now
to study our theta--regularised contact eta function $\vartheta_S$
from its ``zeta'' viewpoint $\eta(S_Q)(s)$.

We first study the odd zeta regularised Chern character of $V$
i.e. the series
\begin{displaymath}
  \ch(V)^{\zeta}_{odd}(s) = \sum_{\lambda \in \spec^*(iT)}
  \frac{\lambda \ch(V_\lambda)}{|\lambda|^{2s+1}} = \sum_{\lambda \in
    \spec^*(iT)} \dim V^x \frac{\lambda e^{\lambda \c}}{|\lambda|^{2s+1} }
\end{displaymath}

We proceed by taking Mellin transform of
$\ch(V)^{\theta}_{odd}(t)$. From Proposition~\ref{prop:ch_theta} its
smooth part reads
\begin{align*}
  \ch(V)^{\theta}_{odd}(t)
  & = \sqrt t \sum_{\lambda \in \spec(iT)} \dim V^x \lambda e^{\lambda \c
    -t \lambda^2} \\ 
  & = \sum_{\lambda \in \spec^*(iT)} \dim V^x \sum_{p=0}^{2k-1}
    \frac{\c^p}{p!} \lambda^{p+1} \sqrt t e^{-t \lambda^2} \,,
\end{align*}
with $\dim N = 4k -2$.  Then one finds by dominated convergence and
direct computation that for $\Re(s) > k$ the following integral
converges and
\begin{displaymath}
  I(s) = \int_0^{+\infty} \ch(V)^{\theta}_{odd}(t) t^{s-1} dt
  = \Gamma(s+\frac{1}{2}) \sum_{p=0}^{2k-1} \frac{\c^p}{p!} \sum_{e^{2i
      \pi x} \in \spec(\rho(f))} Z(p,s,x)
\end{displaymath}
with
\begin{equation}
  \label{eq:48}
  Z(p,s,x) =
  \begin{cases}
    \dim V^x (\zeta(2s-p, x) +(-1)^{p+1} \zeta(2s-p,
    1-x)) & \quad \mathrm{if} \ 0< x <1 \\
    \dim V^1 (1+(-1)^{p+1})\zeta(2s-p) & \quad \mathrm{if}\ x=1
  \end{cases}
\end{equation}
where $\zeta(s,x)$ is Hurwitz zeta function \eqref{eq:37}. Therefore
$I(s)/\Gamma(s + \frac{1}{2})$ defines a meromorphic function with
possible simple poles at $s = \frac{j}{2}$ for $j \in [[ 1,
2k]]$. Hence we get that the series
\begin{align}
  \label{eq:49}
  \ch(V)^{\zeta}_{odd}(s)
  & = \sum_{\lambda \in \spec(iT)*} \lambda
    \frac{\ch(V_\lambda)}{|\lambda|^{2s+1}} = 
    \frac{I(s)}{\Gamma(s+\frac{1}{2})} \\
  & = \sum_{p=0}^{2k-1} \frac{\c^p}{p!} \sum_{e^{2i \pi x} \in
    \spec(\rho(f))} Z(p,s,x)
    \nonumber
\end{align}
are well defined and meromorphic on the same domain.

\smallskip

On the discrete part $\ch(V)^\theta_{odd}(f_j^r)$ as given in
Proposition~\ref{prop:ch_theta}, one finds that for $\Re(s) > 1/2$
\begin{equation}
  \label{eq:50}
  \ch(V)^\zeta_{odd}(f_j^r)(s) =  \sum_{e^{2i\pi x} \in
    \spec(\rho(f))} \chi_\rho(f_j^r\mid V^x) Z(f_j^r,s,x) 
\end{equation}
with
\begin{multline}
  \label{eq:51}
  Z(f_j^r,s,x) =   \\
  \begin{cases}
    e^{-2i \pi rx /\alpha_j} L (e^{-2 i \pi r / \alpha_j}, 2s,x) -
    e^{2i \pi r (1-x)/\alpha_j} L (e^{2 i \pi r / \alpha_j}, 2s,1-x)
    & \ \mathrm{for}\ x \not = 1 \\
    e^{-2i \pi r/\alpha_j} L(e^{-2i \pi r/ \alpha_j},2s,1) -
    L(e^{2i\pi r /\alpha_j},2s,1) & \ \mathrm{for}\ x =1 \,.
  \end{cases}
\end{multline}
where
\begin{displaymath}
  L(z, s ,x) = \sum_{n \geq 0}\frac{z^n}{(n+
    x)^s}
\end{displaymath}
is Lerch zeta function. In our case, $z=e^{\pm 2i \pi r /\alpha_j}$ is
a root of unity and
\begin{equation}
  \label{eq:52}
  L(z,s,x) = \frac{1}{\alpha_j^{s}} \sum_{m=0}^{\alpha_j-1} z^m \zeta(s,
  \frac{m+x}{\alpha_j}) 
\end{equation}
splits into a sum of Hurwitz zeta functions. From
\cite[p.~255]{Apostol}, $\zeta(s,x)$ is analytic with a simple pole at
$s=1$ and residue $1$. Therefore
$\mathrm{Res}_{s=1/2} L(e^{\pm 2i\pi r /\alpha_j}, 2s ,0)=0 $ and the
functions $Z(f_j^r,s,x)$ are entire.

\smallskip

This gives an explicit geometric expression for
$\ch(V)^\zeta_{odd}(s)$ that leads with \eqref{eq:44} and
\eqref{eq:47} to the formula for $\eta(S_Q)(s)$ from
\begin{align}
  \label{eq:53}
  - \eta(S_Q)(s)
  & =  \langle \ch(V)^\zeta_{odd}(s) \wedge
    \mathcal{L}(N)_{orb}, [N] \rangle_{orb}\,. \\
  & =  \langle \ch(V)^\zeta_{odd}(s) \wedge \mathcal{L}(N) , [N_{smooth}]
    \rangle + \sum_i \frac{1}{\alpha_i} \sum_{r=1}^{\alpha_i-1}
    \ch(V)^\zeta_{odd}(f_i^r)(s) \nu(f_i^r)  \,.\nonumber
\end{align}

\smallskip

Specialising at the regular value $s=0$ gives an explicit formula for
the contact eta invariant $\eta(S_Q)(0)$. Let
\begin{displaymath}
  B(t,x) = \frac{t e^{tx}}{e^t-1} = \sum_{n=0}^{+\infty} B_n(x) \frac{t^n}{n!}
\end{displaymath}
be the generating function of Bernoulli polynomials; see
e.g. \cite{Apostol}, and
\begin{displaymath}
  B_{ev}(t,x) = t\frac{  \cosh(t(x-\frac{1}{2}))}{2 \sinh(\frac{t}{2})} =
  \sum_{n=0}^{+\infty} B_{2n}(x) \frac{t^{2n}}{(2n)!} 
\end{displaymath}
its even part in $t$.  We consider the function
\begin{align*}
  (\Delta B_{ev})(t,x)
  & = \frac{1}{t}(B_{ev}(t,x) -1) = \sum_{n=1}^{+\infty} B_{2n}(x)
    \frac{t^{2n-1}}{(2n)!} \\ 
  & = \frac{  \cosh(t(x-\frac{1}{2}))}{2 \sinh(\frac{t}{2})} - \frac{1}{t}\,. 
\end{align*}
In the following statement, we recall that
$V^x = \ker (\rho(f) - e^{2i \pi x} \id)$ for $0< x \leq 1$.
\begin{thm}
  \label{thm:geo-eta-invariant}
  On a CR Seifert manifold, it holds that
  \begin{multline*}
    \eta(S_Q)(0) = 2 \sum_{e^{2i\pi x}\in \spec(\rho(f))} \dim V^x
    \langle
    (\Delta B_{ev})(\c, x) \wedge \mathcal{L}(N), [N]_{smooth} \rangle  \\
    + 2 \sum_{e^{2i\pi x}\in \spec(\rho(f))} \sum_j \frac{1}{\alpha_j}
    \sum_{r=1}^{\alpha_j-1} \chi_\rho(f_j^r \mid V^x) \frac{B(-2 i \pi
      r/ \alpha_j, x)}{-2i \pi r /\alpha_j} \nu(f_j^r) \,.
  \end{multline*}
\end{thm}

This formula generalises an expression given in
\cite[Theorem~8.8]{BHR} in the three dimensional case and with a
trivial representation. There one has
\begin{displaymath}
  \eta(S_Q)(0) = \eta (*D + d_Q \delta_Q)(0) = \eta_0(M, \theta) - 1  \,, 
\end{displaymath}
where
$\eta_0(M, \theta)= \mathrm{FP}_{\varepsilon=0}\,\eta(*_\varepsilon
d)$ is the renormalised eta invariant of $*_\varepsilon d$ for the
sub-Riemannian limit of metrics
$g_\varepsilon = d \theta(\cdot, J \cdot) + \varepsilon^{-1}\theta^2$
when $\varepsilon \rightarrow 0$. From \cite[Section~3]{BHR}, this
also coincides with the adiabatic limit of $\eta(*_\varepsilon d)$ for
$\varepsilon \rightarrow +\infty$ in our fibrered case.

\smallskip

Note that in general dimension, the formula for $\eta(S_Q)(0)$ no
longer shows up global (over $M$) cohomological terms, in contrast to
$\eta(*D)(0)$ by Proposition~\ref{prop:eta-etaS_Q}. This is due to the
fact that the spectrum of $S_Q$ is much more symmetric than the one of
$*D$, that contains copies of zeta functions of Laplacians, leading to
additional cohomological terms at $s=0$.

Also, from \cite{Albin-Quan}, one knows that $\eta(*D)$ compares on
general contact manifolds with the renormalised sub-Riemanian limit
$\eta_0(M,\theta)= FP_{\varepsilon=0} (*_\varepsilon d)$, up to local
terms. It is also the adiabatic limit of the eta invariant in our CR
Seifert situation. This limit has been studied in depth by other
means. Bilding on previous works by Bismut and Cheeger \cite{BC}, Dai
\cite{Dai1991} and Zhang \cite{Zhang} expressed $\eta_0(M, \theta)$
for circle bundles, in the case of a trivial representation and with
no singular points. From \cite[Theorem~0.3]{Dai1991} and
\cite[Theorem~1.7 and 2.5]{Zhang} it reads
\begin{displaymath}
  \eta_0(M,\theta)= 2 \langle \Delta B_{ev}(\c, 1) \wedge \mathcal{L}(N),
  [N]\rangle + \frac{1}{2} \dim  \ker (D_N) + 2 \tau
\end{displaymath}
where $D_N$ is some Dirac operator over $N$ and $\tau $ is the
signature of the ``collapsing spectrum'' in the adiabatic limit, an
integer. Therefore, modulo an half-integer (or an integer if
$\dim \ker D_N$ is even in our particular situation),
$\eta_0(M, \theta)$ coincides with $\eta(S_Q)(0)$ as given in
Theorem~\ref{thm:geo-eta-invariant}, in the case of a trivial
representation and a smooth $N$.

In the adiabatic viewpoint, the term
$2 \langle \Delta B_{ev}(\c, 1) \wedge \mathcal{L}(N), [N]\rangle$
comes from an eta form valued invariant constructed by Bismut and
Cheeger \cite{BC}. It is associated to a Dirac operator for a
superconnection over the fibers. Here, in this smooth case, from our
viewpoint
\begin{displaymath}
  2 \langle \Delta B_{ev}(\c, 1) \wedge \mathcal{L}(N), [N]\rangle =
  \eta(S_Q)(0) \,,
\end{displaymath}
so that this term is interpreted as being the eta invariant of a
second order hypoelliptic differential operator over the whole contact
manifold $M$.

\begin{proof}[Proof of Theorem~\ref{thm:geo-eta-invariant}]
  \
  
  We have to evaluate $Z(0,p,x)$ as given in \eqref{eq:48}. From
  e.g. \cite[p.~264]{Apostol}, we know that
  \begin{displaymath}
    \zeta(-p,x) = - \frac{B_{p+1}(x)}{p +1} \quad \mathrm{for}\ p \in\N \,.
  \end{displaymath}
  Moreover $B_n(1-x) = (-1)^n B_n(x)$, so that $Z(0,x,p) = 0$ for $p$
  even, while for $p$ odd
  \begin{displaymath}
    Z(0,x,p) = - 2 \dim V^x \frac{B_{p+1}(x)}{p+1} \,.
  \end{displaymath}
  Hence from \eqref{eq:49}
  \begin{align*}
    \ch(V)^\zeta_{odd}(0)
    & = -2 \sum_{e^{2i \pi x} \in \spec(f)} \dim V^x \sum_{p=1\,
      p\,odd}^{2k-1} B_{p+1}(x) \frac{\c^p}{(p+1) !}\\
    & = -2 \sum_{e^{2i \pi x} \in \spec(f)} \dim V^x (\Delta B_{ev})(\c, x)
      \,,
  \end{align*}
  as needed. Note that one can replace $\Delta B_{ev}$ by $\Delta B$
  in the formula for $\eta(S_Q)(0)$ since
  $\langle (\Delta B_{odd}(\c, x) \wedge \mathcal{L}(N),
  [N]_{smooth}\rangle = 0$ for dimensional reasons.

  \smallskip We compute the discrete contribution of $f_j^r$. Let
  $z= e^{\pm 2 i \pi r/\alpha_j}$. From $\zeta(0, x) = 1/2 -x$ and
  \eqref{eq:52}, one gets
  \begin{align*}
    L(z,0, x)
    & = \sum_{m=0}^{\alpha_j -1} z^m(\frac{1}{2} -
      \frac{m+x}{\alpha_j}) =- \frac{1}{\alpha_j} \sum_{m=0}^{\alpha_j-1} m
      z^m  \\
    & = \frac{1}{1 -z} \,,
  \end{align*}
  by deriving the identity
  $\dsp \sum_{m=0}^{\alpha_i -1} u^{m+1} = \frac{u(1-u^{\alpha_j})}{1
    -u}$ at $u=z$. Inserting it in \eqref{eq:51} gives the result.
  
\end{proof}

\subsection{The contact eta function from the dynamical viewpoint}
\label{sec:contact-eta-function}

\ \smallskip

We eventually work out the dynamical aspect of the eta function and
eta invariant. We start with the smooth part.  From
Proposition~\ref{prop:ch_theta} we have that for some $a>0$
\begin{displaymath}
  \ch(V)^\theta_{odd}(t) = O(e^{-at}) \quad \mathrm{when} \quad t
  \rightarrow +\infty \,.
\end{displaymath}
Due to terms of type $e^{-t(\ell(\gamma) + i \c)^2/4t}$ in the
dynamical expression of $\ch(V)^\theta_{odd}(t)$, its behaviour when
$t \rightarrow 0^+$ is unclear, unless if
$\|\c\| < m =\min (\ell(f), \ell(f_i))$. In that case, one has for
some $b>0$
\begin{displaymath}
  \ch(V)^\theta_{odd} (t) - \dim V \frac{i \ell (f) \c}{4t \sqrt{\pi}}
  e^{\c^2/4t} = 0(e^{-b/t}) \,,
\end{displaymath}
with
\begin{displaymath}
  \frac{\c}{4t} e^{\c^2/4t} = \sum_{p=0}^{k-1} \frac{\c^{2p+1}}{p! (4t)^{p+1}} \,.
\end{displaymath}
To get around this, we first replace $\c$ by $u \c$ for small enough
$u$ so that the divergence when $t\rightarrow 0$ only comes from the
trivial constant orbit. There previous formulae on
$\ch(V)^\theta_{odd}$ hold as well with the advantage that
\begin{displaymath}
  f(t) = \ch(V)^\theta_{odd}(t) -  \dim V \frac{i u \c \ell(f)}{4t \sqrt \pi}
  e^{u^2 \c^2/4t} \mathbf{1}_{]0,1]}(t)
\end{displaymath}
has an entire Mellin transform over $\C$. Hence, proceeding as in
Section~\ref{sec:proof-theor-refthm:z} we first find that for
$\Re(s) > k$
\begin{align*}
  I(s) = \int_0^{+\infty} f(t) t^{s-1} dt
  & = \Gamma(s + \frac{1}{2})  \ch(V)^\zeta_{odd}(s)  \\
  & - \dim V \sum_{p=0}^{k-1}\frac{i \ell(f)(u \c)^{2p+1}}{p! 4^{p+1} \sqrt
    \pi} \times \frac{1}{s-p-1} \,,
\end{align*}
While writing
\begin{displaymath}
  f(t) = \ch(V)^\theta_{odd}(t) - \dim V \frac{i \ell(f) u \c}{4t \sqrt \pi}
  e^{u^2 \c^2/4t} + \dim V \frac{i \ell(f)u \c }{4t \sqrt \pi}
  e^{u^2 \c^2/4t} \mathbf{1}_{]1,+\infty]}(t)
\end{displaymath}
and using the dynamical expression for $\ch(V)^\theta_{odd}(t)$ one
has for $\Re(s) < 0$
\begin{align*}
  I(s) = \Gamma(1 - s)\bigl(
  & \sum_{n \in \Z^*}- \frac{i\ell(f)}{4^s\sqrt \pi}
    \chi_\rho(f^n) (\ell(f^n) + iu\c) \bigl( (\ell(f^n) + iu\c)^2\bigr)^{s-1}
    \bigr)\\
  &  -  \dim V\sum_{p=0}^{k-1}\frac{i \ell(f) (u \c)^{2p+1}}{p! 4^{p+1}
    \sqrt \pi} \times \frac{1}{s-p-1} \,.
\end{align*}
Here $z^s$ denotes the principal branch of the power function. We get
then the identity of meromorphic functions through analytic
continuation
\begin{align*}
  \Gamma(s+1/2) \ch(V)^\zeta_{odd}(s)
  & = \\
  & \Gamma(1-s) \sum_{n \in \Z^*}-
    \frac{i\ell(f)}{4^s\sqrt \pi} \chi_\rho(f^n) (\ell(f^n) + iu\c) \bigl(
    (\ell(f^n) + iu\c)^2\bigr)^{s-1} 
\end{align*}
That extends continuously to $u = 1$ by \eqref{eq:49}.

One finds more easily the discrete dynamical contribution of $f_j^r$
due to rapid decay of $\ch(V)^\theta_{odd}(f_j^r)$ both at $0$ and
$+\infty$ from Proposition~\ref{prop:ch_theta}. It gives that
\begin{multline*}
  \Gamma(s+1/2) \ch(V)^\zeta_{odd}(f_j^r)(s) =
  \\
  \Gamma(1-s) \sum_{n \in \Z}- \frac{i\alpha_j \ell(f_j)}{4^s\sqrt
    \pi} \chi_\rho(f_j^{r + n \alpha_j}) \ell(f_j^{r + n \alpha_j})
  \bigl(\ell(f_j^{r + n \alpha_j})^2\bigr)^{s-1}
\end{multline*}
is an entire function. Gathering this with \eqref{eq:53} we get the
following.
\begin{thm}
  \label{thm:zeta-eta-dynamique}
  The following identity defines a meromorphic function $\varphi$ with
  simple poles at $s = 1, \cdots, k$
  \begin{displaymath}
    \varphi(s)= \Gamma(s+\frac{1}{2}) \eta(S_Q)(s) = \frac{\Gamma(1-s)}{4^s
      \sqrt\pi} \sum_{\gamma\not=0} \chi_\rho(\gamma) \eta(\gamma)(s)
  \end{displaymath}
  with
  \begin{align*}
    \eta(\gamma)(s)
    & =
      i\ell(f) \langle(\ell(\gamma) + i\c) \bigl( (\ell(\gamma) +
      i\c)^2\bigr)^{s-1}\wedge \mathcal{L}(N), [N_{smooth}] \rangle \
      \mathrm{if}\
      \gamma= f^n \,,\\
    \intertext{  and}
    \eta(\gamma)(s)
    & = i \ell(f_j) \ell(\gamma) (\ell(\gamma)^2)^{s-1}
      \nu(\gamma) \ \mathrm{if} \ \gamma = f_j^n ,\ n \not\equiv 0 \mod
      \alpha_j \,. 
  \end{align*}

  Moreover
  \begin{displaymath}
    \mathrm{Res}_{s = p}(\varphi) = \dim V \frac{i\ell(f)}{\sqrt \pi (p-1) !
      4^p} \langle\c^{2p-1}\wedge \mathcal{L}(N), [N_{smooth}]\rangle \,.
  \end{displaymath}
\end{thm}

The function being regular at $s=0$ we get the dynamical formula for
the eta invariant.
\begin{cor}
  \label{cor:eta-dynamic}
  On a CR Seifert manifold, one has
  \begin{displaymath}
    \eta(S_Q)(0) = \sum_{\gamma\not=0} \chi_\rho(\gamma) \eta(\gamma)(0)
  \end{displaymath}
  with
  \begin{displaymath}
    \eta(\gamma)(0)  =
    \frac{i\ell(f)}{\pi} \langle\frac{1}{\ell(\gamma) +
      i\c} \wedge \mathcal{L}(N), [N_{smooth}] \rangle \quad \mathrm{if}\
    \gamma= f^n \,,
  \end{displaymath}
  and
  \begin{displaymath}
    \eta(\gamma)(0) = \frac{i \ell(f_j)}{\pi
      \ell(\gamma)}\nu(\gamma) \quad\mathrm{if} \ \gamma = f_j^n ,\ n
    \not\equiv 0 \mod \alpha_j \,. 
  \end{displaymath}
\end{cor}

This is a decomposition of the eta invariant into its dynamical
``atoms''.  These dynamical series are \emph{a priori} formal
expressions coming from analytic continuation. However they can be
turned into convergent ones. The smooth contribution is actually an
absolutely convergent series using the $\c \leftrightarrow - \c$
symmetry; see Remark~\ref{rem:sigma_symmetry}. It comes as a limit of
the smooth dynamical expression in
Theorem~\ref{thm:zeta-eta-dynamique} when $s \rightarrow 0^-$. The
discrete contribution are semi-convergent series when gathering the
orbit contributions of $\gamma = f_j^{r + n \alpha_j}$ and
$\overline\gamma = f_j^{r - n \alpha_j}$. Using Abel's lemma one sees
that it is also the limit coming from the corresponding dynamical
expression of Theorem~\ref{thm:zeta-eta-dynamique} when
$s \rightarrow 0^-$.

\bigskip \bigskip

\bibliographystyle{abbrv}

\bibliography{torsion_eta}

\begin{thebibliography}{10}

\bibitem{Albin-Quan}
P.~Albin and H.~Quan.
\newblock Sub-{Riemannian} limit of the differential form heat kernels of
  contact manifolds.
\newblock {\em Int. Math. Res. Not.}, 2022(8):5818--5881, 2022.

\bibitem{Apostol}
T.~M. Apostol.
\newblock {\em Introduction to analytic number theory}.
\newblock Undergraduate Texts Math. Springer, Cham, 1976.

\bibitem{Atiyah-Bott}
M.~F. Atiyah and R.~Bott.
\newblock A {Lefschetz} fixed point formula for elliptic complexes. {II}:
  {Applications}.
\newblock {\em Ann. Math. (2)}, 88:451--491, 1968.

\bibitem{At-Pa-SiI}
M.~F. Atiyah, V.~K. Patodi, and I.~M. Singer.
\newblock Spectral asymmetry and {R}iemannian geometry. {I}.
\newblock {\em Math. Proc. Cambridge Philos. Soc.}, 77:43--69, 1975.

\bibitem{At-Pa-SiIII}
M.~F. Atiyah, V.~K. Patodi, and I.~M. Singer.
\newblock Spectral asymmetry and {R}iemannian geometry. {III}.
\newblock {\em Math. Proc. Cambridge Philos. Soc.}, 79(1):71--99, 1976.

\bibitem{BHR}
O.~Biquard, M.~Herzlich, and M.~Rumin.
\newblock Diabatic limit, eta invariants and {C}auchy-{R}iemann manifolds of
  dimension $3$.
\newblock {\em Ann. Sci. Ecole Norm. Sup. (4)}, 40(4):589--631, 2007.

\bibitem{BC}
J.-M. Bismut and J.~Cheeger.
\newblock {$\eta$}-invariants and their adiabatic limits.
\newblock {\em J. Amer. Math. Soc.}, 2(1):33--70, 1989.

\bibitem{Cheeger}
J.~Cheeger.
\newblock Analytic torsion and the heat equation.
\newblock {\em Ann. of Math. (2)}, 109(2):259--322, 1979.

\bibitem{Dai1991}
X.~Dai.
\newblock Adiabatic limits, nonmultiplicativity of signature, and {L}eray
  spectral sequence.
\newblock {\em J. Amer. Math. Soc.}, 4(2):265--321, 1991.

\bibitem{Fried87}
D.~Fried.
\newblock Lefschetz formulas for flows.
\newblock In {\em The Lefschetz centennial conference, Part III (Mexico City,
  1984)}, volume~58 of {\em Contemp. Math.}, pages 19--69. Amer. Math. Soc.,
  Providence, RI, 1987.

\bibitem{Fried}
D.~Fried.
\newblock Counting circles.
\newblock In {\em Dynamical systems (College Park, MD, 1986--87)}, volume 1342
  of {\em Lecture Notes in Math.}, pages 196--215. Springer, Berlin, 1988.

\bibitem{Kawasaki}
T.~Kawasaki.
\newblock The {R}iemann-{R}och theorem for complex {$V$}-manifolds.
\newblock {\em Osaka J. Math.}, 16(1):151--159, 1979.

\bibitem{Kitaoka}
A.~Kitaoka.
\newblock Ray-{Singer} torsion and the {Rumin} {Laplacian} on {Lens} spaces.
\newblock {\em SIGMA, Symmetry Integrability Geom. Methods Appl.}, 18:paper
  091, 16, 2022.

\bibitem{Muller}
W.~M{\"u}ller.
\newblock Analytic torsion and {$R$}-torsion of {R}iemannian manifolds.
\newblock {\em Adv. Math.}, 28(3):233--305, 1978.

\bibitem{Ornea-Verbitsky}
L.~Ornea and M.~Verbitsky.
\newblock Sasakian structures on {CR}-manifolds.
\newblock {\em Geom. Dedicata}, 125:159--173, 2007.

\bibitem{RS}
D.~B. Ray and I.~M. Singer.
\newblock {$R$}-torsion and the {L}aplacian on {R}iemannian manifolds.
\newblock {\em Adv. Math.}, 7:145--210, 1971.

\bibitem{Rumin94}
M.~Rumin.
\newblock Formes diff\'erentielles sur les vari\'et\'es de contact.
\newblock {\em J. Differential Geom.}, 39(2):281--330, 1994.

\bibitem{Rumin00}
M.~Rumin.
\newblock Sub-{R}iemannian limit of the differential form spectrum of contact
  manifolds.
\newblock {\em Geom. Funct. Anal.}, 10(2):407--452, 2000.

\bibitem{Rumin-Seshadri}
M.~Rumin and N.~Seshadri.
\newblock Analytic torsions on contact manifolds.
\newblock {\em Ann. Inst. Fourier (Grenoble)}, 62(2):727--782, 2012.

\bibitem{Weil}
A.~Weil.
\newblock Introduction {\`a} l'{\'e}tude des vari{\'e}t{\'e}s
  k{\"a}hl{\'e}riennes.
\newblock Actualit{\'e}s {Scientifiques} et {Industrielles}, {No}. 1267,
  {Publications} de l'{Institut} de {Math{\'e}matique} de l'{Universit{\'e}} de
  {Nancago}, {VI}. {Paris}: {Hermann} \& {Cie}. 175 pp. (1958)., 1958.

\bibitem{WW}
E.~T. Whittaker and G.~N. Watson.
\newblock {\em A course of modern analysis. {A}n introduction to the general
  theory of infinite processes and of analytic functions: with an account of
  the principal transcendental functions}.
\newblock Fourth edition. Reprinted. Cambridge University Press, New York,
  1962.

\bibitem{Zhang}
W.~P. Zhang.
\newblock Circle bundles, adiabatic limits of {$\eta$}-invariants and {R}okhlin
  congruences.
\newblock {\em Ann. Inst. Fourier (Grenoble)}, 44(1):249--270, 1994.

\end{thebibliography}


-----------------------------------------------------

\end{document}